\documentclass[french, 10pt, a4paper, oneside,bibliography=totocnumbered]{scrartcl}
\usepackage[T1]{fontenc} 
\usepackage[latin1]{inputenc} 
\usepackage{mathptmx} 
\usepackage{lmodern} 
\usepackage{amsthm} 
\usepackage{amsmath} 
\usepackage{amsfonts} 
\usepackage{amssymb} 
\usepackage{mathrsfs}
\usepackage[left=3.10cm, right=3.10cm, top=3.10cm, bottom=3.10cm]{geometry} 
\usepackage[ngerman,german,english]{babel} 
\usepackage{paralist} 
\usepackage[colorlinks=true]{hyperref}
\hypersetup{linktocpage}
\usepackage{etoolbox}

\setkomafont{sectionentry}{\normalsize}
\RedeclareSectionCommand[tocbeforeskip=1pt]{section}

\setkomafont{section}{\normalfont\Large\textbf}
\setkomafont{subsection}{\normalfont\large\textbf}
\setkomafont{paragraph}{\normalfont\textbf}

\renewenvironment{abstract}
{\begin{center}
		\textbf{Abstract}
	\end{center}
	\list{}{ 
		\setlength{\leftmargin}{0.05\textwidth}
		\setlength{\rightmargin}{\leftmargin}
	}
	\item\relax} 
{\endlist}

\newenvironment{keywords}
{\begin{trivlist}\item[]{\bfseries Keywords.}}
	{\end{trivlist}}

\newenvironment{subclass}
{\begin{trivlist}\item[]{\bfseries Mathematical Subject Classification:}}
	{\end{trivlist}}

\newenvironment{acknowledgements}
{\begin{trivlist}\item[]{\bfseries Acknowledgements:}}
	{\end{trivlist}}

\theoremstyle{plain} 
\newtheorem{thm}{Theorem}[section] 

\newtheorem{lem}[thm]{Lemma}
\newtheorem{pro}[thm]{Proposition}

\theoremstyle{definition}
\newtheorem{defi}[thm]{Definition}
\newtheorem{rem}[thm]{Remark}

\newcommand{\N}{\mathbb{N}}

\newcommand{\R}{\mathbb{R}}

\newcommand{\norm}[2][]{\left\|#2\right\|_{#1}}
\newcommand{\tnorm}[2][]{||| #2|||_{#1}}

\newcommand{\dualbra}[2]{\left\langle #1,#2 \right\rangle}

\newcommand{\Lpnorm}[2][]{\ifthenelse{\equal{#1}{}}{\norm{#2}_{L^p}}{\norm{#2}_{L^p(#1)}}}
\newcommand{\Hknorm}[2][]{\ifthenelse{\equal{#1}{}}{\norm{#2}_{H^k}}{\norm{#2}_{H^k(#1)}}}

\newcommand{\set}[1]{\left\lbrace #1 \right\rbrace}

\renewcommand{\ker}{\text{\normalfont ker\,}}
\newcommand{\trace}{\text{\normalfont tr\,}}
\newcommand{\sgn}{\mathrm{sgn}}

\renewcommand{\div}{\text{\normalfont div}}

\newcommand{\QV}[2][]{\ifthenelse{\equal{#1}{}}{\langle #1 \rangle}{\langle #1,#2 \rangle}}

\newcommand{\ind}{\mathbb{I}}

\newcommand{\Lin}{\mathscr{L}}

\makeatletter
\newcounter{author}
\renewcommand*\author[1]{%
	\stepcounter{author}%
	\ifnum\c@author=1
	\gdef\@author{#1}%
	\else
	\xdef\@author{\unexpanded\expandafter{\@author\and#1}}%
	\fi
	\csgdef{author@\the\c@author}{#1}}
\newcommand*\email[1]{%
	\csgdef{email@\the\c@author}{#1}}
\newcommand*\address[1]{%
	\csgdef{address@\the\c@author}{#1}}
\AtEndDocument{%
	\xdef\author@count{\the\c@author}%
	\c@author=1
	\print@authors}
\newcommand*\print@authors{%
	\ifnum\c@author>\author@count
	\else
	\print@author{\the\c@author}%
	\advance\c@author by 1
	\expandafter\print@authors
	\fi}
\newcommand*\print@author[1]{%
	\par\medskip
	\begin{tabular}{@{}l@{}}%
		\textsc{\csuse{author@#1}}\\
		\csuse{address@#1}\\
		\textit{E-mail address}: \csuse{email@#1}
\end{tabular}}
\makeatother

\title{\textrm{\textbf{\Large Longtime behavior for homoenergetic solutions in the collision dominated regime for hard potentials}}}

\author{\large Bernhard Kepka}
\address{University of Bonn, Institute for Applied Mathematics \\ Endenicher Allee 60 \\ D-53115 Bonn \\ GERMANY}
\email{kepka@iam.uni-bonn.de}

\date{\normalsize \today}

\begin{document}

\maketitle

\begin{abstract}
	In this paper, we consider a particular class of solutions to the Boltzmann equation which are referred to as homoenergetic solutions. They describe the dynamics of a dilute gas due to collisions and the action of either a shear, a dilation or a combination of both. We prove that solutions with initially high temperature remain close and converge to a Maxwellian distribution with temperature going to infinity. Furthermore, we give precise asymptotic formulas for the temperature. This local stability result is a consequence of a dominant shear and the homogeneity $ \gamma>0 $ of the collision operator with respect to relative velocities. The proof relies on an ansatz which is motivated by a Hilbert-type expansion. We consider both non-cutoff and cutoff kernels.
	\begin{keywords}
		Boltzmann equation, Homoenergetic solutions, Long-range interactions, Non-equilibrium, Hard potentials
	\end{keywords}
	\begin{subclass}
		35Q20, 82C40, 35C20
	\end{subclass}
	
	\begin{acknowledgements}
		The author thanks Juan J.L. Vel{\'a}zquez for the suggestion of the problem and helpful discussions. The author has been funded by the Deutsche Forschungsgemeinschaft (DFG, German Research Foundation) under Germany's Excellence Strategy -- EXC-2047/1 -- 390685813 as a member of the Bonn International Graduate School of Mathematics (BIGS). Furthermore, the work was supported by the DFG through the collaborative research centre \textit{The mathematics of emerging effects} (CRC 1060, Project-ID 211504053).
	\end{acknowledgements}
\end{abstract}

\tableofcontents
\newpage
\section{Introduction}
The inhomogeneous Boltzmann equation is given by 
\begin{align}\label{eq:inhomBoltzmannEq}
	\partial_t f + v\cdot \nabla_x f = Q(f,f),
\end{align}
where $ f=f(t,x,v):[0,\infty)\times\R^3\times\R^3\rightarrow[0,\infty) $ denotes the one-particle distribution of a dilute gas in whole space. In this paper, we will restrict ourselves to the physically most relevant case of three dimensions, although our study can be extended to dimensions $ N\geq 3 $ without any additional difficulties.

In \eqref{eq:inhomBoltzmannEq} the bilinear collision operator has the form
\begin{align*}
	Q(f,g) = \int_{\R^3}\int_{S^2}B(|v-v_*|,n\cdot \sigma)(f'_*g'-f_*g)d\sigma dv_*,
\end{align*}
where $ n=(v-v_*)/|v-v_*| $ and $ f'_* = f(v'_*) $, $ g' = g(v') $, $ f_*=f(v_*) $, with the pre-collisional velocities $ (v,v_*) $ resp. post-collisional velocities $ (v',v'_*) $. Here, we use the $ \sigma $-representation of post-collisional velocities, i.e. for $ \sigma \in S^2 $
\begin{align*}
	\begin{split}
		v' = \dfrac{v+v_*}{2}+\dfrac{|v-v_*|}{2}\sigma, \qquad
		v'_* = \dfrac{v+v_*}{2}-\dfrac{|v-v_*|}{2}\sigma.
	\end{split}
\end{align*}
Recall that the collision operator satisfies
\begin{align*}
	\int_{\R^3} Q(f,f) \varphi(v)dv=0, \quad \varphi(v)=1,\,v_1,\, v_2,\, v_3, \, |v|^2
\end{align*}
which correspond to the conservation of mass, momentum and energy. We refer to \cite{Cercignani1988BoltzmannEqApplication,Villani2002Review} for an introduction into the physical and mathematical theory of the Boltzmann equation \eqref{eq:inhomBoltzmannEq}.

The collision kernel is given by $ B(|v-v_*|,n\cdot \sigma) $ and it can be obtained from an analysis of the binary collisions of the gas molecules. For instance, power law potentials $ 1/r^{q-1} $ with $ q>2 $ lead to (see e.g. \cite[Sec. II.5]{Cercignani1988BoltzmannEqApplication})
\begin{align}\label{eq:CollisinKernel}
	B(|v-v_*|,n\cdot \sigma)=|v-v_*|^{\gamma}\, b(n\cdot \sigma), \quad \gamma=(q-5)/(q-1),
\end{align}
where $ b:[-1,1)\rightarrow [0,\infty) $ has a non-integrable singularity of the form
\begin{align}\label{eq:angularSing}
	\sin\theta \, b(\cos\theta) \sim \theta^{-1-2/(q-1)}= \theta^{-1-2s}, \quad \text{as } \theta\to 0, \quad s=\dfrac{1}{q-1},
\end{align}
where $ \cos\theta = n\cdot \sigma $, with $ \theta $ being the deviation angle. It is customary to classify the collision kernels according to their homogeneity $ \gamma $ with respect to relative velocities $ |v-v_*| $. There are three cases: \textit{hard potentials} $ (\gamma>0) $, \textit{Maxwell molecules} $ (\gamma=0) $ and \textit{soft potentials} ($ \gamma<0 $). Furthermore, collision kernels with an angular singularity of the form \eqref{eq:angularSing} are called non-cutoff kernels. This singularity reflects the fact that for power law interactions the average number of \textit{grazing collisions}, i.e. collisions with $ v\approx v' $, diverges. In the main part of the paper, we consider non-cutoff collision kernels for hard potentials $ \gamma>0 $.

In this paper, we study a particular class of solutions to \eqref{eq:inhomBoltzmannEq}, namely the so-called \textit{homoenergetic solutions}, which have been analyzed in particular in \cite{VelazqJames2019LongAsympHomoen}. There conjectures for the longtime asymptotics of solutions have been formulated. Here, we give a rigorous proof of these conjectures in the case of hard potentials $ \gamma>0 $.

\subsection{Homoenergetic solutions}
Our study concerns solutions to \eqref{eq:inhomBoltzmannEq} with the form
\begin{align}\label{eq:ansatzHomoenergSol}
	f(t,x,v)=g(t,v-L(t)x), \quad w=v-L(t)x,
\end{align} 
for $ L(t)\in\R^{3\times 3} $ and a function $ g=g(t,w):[0,\infty)\times\R^3\rightarrow[0,\infty) $ to be determined. One can check that in general, solutions to \eqref{eq:inhomBoltzmannEq} with the form \eqref{eq:ansatzHomoenergSol} exist for a large class of functions $ g $ if and only if $ g $ and $ L $ satisfy
\begin{align}\label{eq:homenergBE}
\begin{split}
	\partial_t g - L(t)w\cdot\nabla_w g &= Q(g,g), \\
	\dfrac{d}{dt}L(t) + L(t)^2 &= 0.
\end{split}
\end{align}
The second equation is used to reduce the variables $ (t,x,v) $ to $ (t,w) $. In particular, the collision operator acts on $ g $ only through the variable $ w $. The second equation can be solved explicitly $ L(t)=L(0)(I+tL(0))^{-1} $. Note that the inverse matrix might not be defined for all times, although this situation will not be considered here.

Solutions to \eqref{eq:homenergBE} are called \textit{homoenergetic solutions} and were introduced by Galkin 
\cite{GalkinOnSolKineticEqBoltzmann} and Truesdell \cite{Truesdell1956PressureFluxII}. They studied their properties in the case of Maxwell molecules via moment equations, a method known since the work by Truesdell and Muncaster \cite{TruesdellMuncasterFundamentalsMaxwell}, see \cite{Galkin1958ClassOfSolution,Galkin1964OneDimUnsteadySol,GalkinOnSolKineticEqBoltzmann,Galkin1966ExactSol,Truesdell1956PressureFluxII}. More recently, this method has also been used in \cite{GarzoSantos2003KineticTheory} (and references therein) for homoenergetic solutions of the Boltzmann equation as well as other kinetic models like BGK. The case of mixtures of gases has been studied there as well.

The well-posedness of \eqref{eq:homenergBE} in the case cutoff hard potentials was proved by Cercignani \cite{Cercignani1989ExistenceHomoenergetic}. See also \cite{Cercignani2001ShearFlowGranularMaterial,CercignaniBoltzmannApproachShearFlow} for a study of shear flow for granular media. Homoenergetic solutions for the two-dimensional Boltzmann equation as well as for a class of Fokker-Planck equations have been studied in \cite{Matthies2019}.

A systematic analysis of the large time behavior of solutions to \eqref{eq:homenergBE} for kernels with arbitrary homogeneities has been undertaken in \cite{BobylevVelazq2020SelfsimilarAsymp,VelazqJames2019LongAsympHomoen,JamesVelazq2019SelfsimilarProfilesHomoenerg,JamesVelazq2019LongtimeAsymptoticsHypDom}. In particular, self-similar solutions have been studied in \cite{BobylevVelazq2020SelfsimilarAsymp,Liu2020BoltzmannUniformShearFlow,JamesVelazq2019SelfsimilarProfilesHomoenerg} for cutoff Maxwell molecules, respectively in \cite{Kepka2021SelfSimilar} for non-cutoff Maxwell molecules. On the other hand, in \cite{VelazqJames2019LongAsympHomoen} the longtime behavior for non-Maxwellian molecules has been analyzed in the case that the collision operator is dominant over the drift term. This suggest that solutions approach the equilibrium distribution. However, since the temperature is not conserved, the equilibrium distribution has a temperature varying with time. The core of the analysis was an adaptation of a Hilbert expansion in order to determine the behavior of the temperature for large times. However, the arguments in \cite{VelazqJames2019LongAsympHomoen} are formal computations and no rigorous proofs are given.

The main contribution of this paper is to give rigorous proofs of the longtime behavior and the asymptotics of the temperature. Here, we study the case of non-cutoff interactions with hard potentials $ \gamma>0 $. As a result, we verify the conjectures in \cite{VelazqJames2019LongAsympHomoen} for such collision kernels.

\subsection{Linearized collision operator and Hilbert-type expansion}
\paragraph{Linearized collision operator.} In order to recall the formal arguments in \cite{VelazqJames2019LongAsympHomoen} let us introduce the linearized collision operator given by
\begin{align*}
		\Lin h = -Q(h,\mu)-Q(\mu,h) = -\int_{\R^3}\int_{S^2} B(|v-v_*|,n\cdot \sigma)\left[ \mu'_*h'+h'_*\mu'-\mu_* h-\mu h_* \right] \, d\sigma dv_*.
\end{align*}
Here, $ \mu $ denotes the Maxwellian
\begin{align*}
	\mu(v)= \dfrac{1}{(2\pi)^{3/2}}e^{-|v|^2/2},
\end{align*} 
which is the (up to a change of the mass, momentum and energy) unique equilibrium solution of the homogeneous Boltzmann equation. The operator $ \Lin $ on $ L^2(\mu^{-1/2}) $ or equivalently the operator $ Lh=\mu^{-1/2}\Lin(\mu^{1/2}h) $ on $ L^2(\R^3) $ has been extensively studied in the literature, see \cite{AlexandreEtAl2011GlobalExistenceFullRegBoltzmannEq,BarangerMouhot2005ExplicitSpectralGapEstimatesLinBoltzmannLandauOp,GressmanStrain2011GlobalClassicalSolBoltzmannEq,Klaus1977BoltzmannCollOp,Mouhot2006ExplicitCoercivityEstLinBoltzmannLandauOp,MouhotStrain2007SpectralGapCoercivityEstLinBoltzmannOp,Pao1974BoltzmannCollOp} and references therein. It is known that $ \Lin $ is a non-negative, self-adjoint operator on $ L^2(\mu^{-1/2}) $. Furthermore, it has a spectral gap if and only if $ \gamma +2s\geq0 $, see \cite{GressmanStrain2011GlobalClassicalSolBoltzmannEq}. Here, $ s\in(0,1) $ is measuring the angular singularity as in \eqref{eq:angularSing} and $ \gamma $ is the homogeneity of the kernel \eqref{eq:CollisinKernel} with respect to $ |v-v_*| $. The case of cutoff kernels is included by setting $ s=0 $. Let us recall that the kernel of $ \Lin $ is given by
\begin{align*}
		\ker\Lin = \text{span} \left\lbrace \mu,\,  v_1\mu, \, v_2\mu, \, v_3\mu, \, |v|^2\mu \right\rbrace.
\end{align*}
This is related to the conservation of mass, momentum and energy of the collision operator. In the case $ \gamma\in(0,1), \, s\in(0,1/2) $, one can show that $ e^{-t \Lin}h $, with $ h $ in weighted $ L^1 $-spaces, approaches $ \ker \Lin $ exponentially fast, see \cite{Tristani2014ExponentialConvergenceEquilHomogBoltz} or Lemma \ref{lem:LinearDecay} below. This relies on the more general framework in \cite{Mouhot2017FactorizationNonSymOper,Mouhot2006RateConvergenceEquilibSpatiallyHomogBoltzmannEq}.

For our analysis, the space $ L^2(\mu^{-1/2}) $, on which $ \Lin $ is self-adjoint and has a spectral gap, is inconvenient, since the exponential decay at infinity $ |v|\to\infty $ is a priori not preserved by equation \eqref{eq:homenergBE}. Thus, we make use of the result in weighted $ L^1 $-spaces. 

\paragraph{Hilbert-type expansion for homoenergetic solutions.} Our goal is to give a rigorous proof of conjectures in \cite{VelazqJames2019LongAsympHomoen} in the case $ \gamma>0 $. For $ \gamma>0 $ they considered matrices $ L_t=L_0(I+tL_0)^{-1} $ which have one of the following asymptotic forms as $ t\to \infty $, assuming $ \det(I+tL(0))>0 $ for all $ t\geq0 $:
\begin{enumerate}[(i)]
	\item Simple shear:
	\begin{align}\label{eq:SimpleShear}
		L_t = \left( \begin{array}{ccc}
			0 & K & 0 \\ 
			0 & 0 & 0 \\ 
			0 & 0 & 0
		\end{array} \right), \quad K\neq 0.
	\end{align}
	
	\item Simple shear with decaying planar dilatation/shear:
	\begin{align}\label{eq:SimpShearDecPlanarDil}
		L_t = \left( \begin{array}{ccc}
			0 & K_2 & 0 \\ 
			0 & 0 & 0 \\ 
			0 & 0 & 0
		\end{array} \right) + \dfrac{1}{1+t} \left( \begin{array}{ccc}
			0 & K_1K_3 & K_1 \\ 
			0 & 0 & 0 \\ 
			0 & K_3 & 1
		\end{array} \right) + \mathcal{O}\left( \dfrac{1}{(1+t)^2} \right) , \quad K_2\neq 0.
	\end{align}
	
	\item Combined orthogonal shear:
	\begin{align}\label{eq:CombOrthShear}
		L_t = \left( \begin{array}{ccc}
			0 & K_3 & K_2-tK_1K_3 \\ 
			0 & 0 & K_1 \\ 
			0 & 0 & 0
		\end{array} \right), \quad K_1K_3\neq 0.
	\end{align}
\end{enumerate}
The above nomenclature was used in \cite[Theorem 3.1]{JamesVelazq2019SelfsimilarProfilesHomoenerg}, where one can find all possible asymptotic forms of $ L_t $ as $ t\to \infty $ under the assumption $ \det(I+tL_0)>0 $ for all $ t\geq0 $. Note that we used the notation $ L_t $ to abbreviate the time-dependence $ L(t) $ as will be done throughout the paper.

For convenience, let us recall the formal asymptotics in \cite{VelazqJames2019LongAsympHomoen}, which can be seen as a variation of a Hilbert expansion. This allows to construct solutions to \eqref{eq:homenergBE} which, after some change of variables, remain close and converge to a Maxwellian distribution.

For the relevant rescaling let us recall the mass $ \rho_t $, momentum $ V_t $ and temperature $ T_t $ of $ g_t $, given by
\begin{align*}
	\rho_t = \int_{\R^3} g_t(w)\, dw, \quad \rho_tT_t = \int_{\R^3} w g_t(w)\, dw, \quad \rho_tT_t = \dfrac{1}{3} \int_{\R^3}\left| w-V _t\right|^2g_t(w)\, dw.
\end{align*}
(Strictly speaking, the term defining $ T_t $ above is $ 2/3 $ times the internal energy. However, this is related to the temperature via the Boltzmann constant.) Using \eqref{eq:homenergBE} one can show that
\begin{align}\label{eq:MacroscopicQuantities}
		\rho_t'= -\trace L_t \, \rho_t, \quad V_t' = -L_tV_t, \quad T_t' = - \dfrac{2}{3\rho_t} \int_{\R^3} (w-V)\cdot L_t(w-V)\, g_t(w)\, dw.
\end{align}
The first two equations can be solved explicitly. We then introduce a rescaling which sets mass to one, momentum to zero and temperature to one. Define $ f_t(v)= g_t(v\beta_t^{-1/2}+V_t)\beta_t^{-3/2}\rho_t^{-1} $, where we used the inverse temperature $ \beta_t=T_t^{-1} $. As a consequence we have
\begin{align}\label{eq:Normalization}
	\int_{\R^3} f_t(v)\, dv=1, \quad \int_{\R^3} vf_t(v)\, dv=0, \quad \dfrac{1}{3}\int_{\R^3}|v|^2 f_t(v)\, dv=1.
\end{align}
Furthermore, we obtain with \eqref{eq:homenergBE} and \eqref{eq:MacroscopicQuantities} the equations
\begin{align}\label{eq:IndtrodModelHomoengBE}
	\begin{split}
		\partial_t f &= \div \left( \left( L_t-\alpha_t\right) v \, f\right)  + \rho_t\, \beta_t^{-\gamma/2}\, Q(f,f), \quad f(0,\cdot)=f_0(\cdot),
		\\
		\beta_t &= \beta_0\exp \left( 2\int_{0}^t\alpha_s\, ds \right),
		\quad
		\alpha_t:=\dfrac{1}{3}\int v\cdot L_tv \, f_t(v)\,dv,
		\\
		\rho_t &= \exp\left( -\int_0^t \trace L_s \, ds \right).
	\end{split}
\end{align}
Note that we set $ \rho_0=1 $, whereas the initial inverse temperature is given by $ \beta_0 $. The equation for the inverse temperature is a consequence of \eqref{eq:MacroscopicQuantities} yielding
\begin{align}\label{eq:IntrodInvTempEq}
	\dfrac{\beta_t'}{2\beta_t}= \alpha_t = \dfrac{1}{3}\int v\cdot L_tv \, f_t(v)\,dv.
\end{align}
 Furthermore, observe that the momentum $ V_t $ does not appear in the evolution equation for $ f $ due to the translation invariance of \eqref{eq:homenergBE}.

Our analysis is concerned with the the longtime behavior of $ f_t $ and we consider here the collision-dominated behavior. This is the case $ \eta_t:=\rho_t\, \beta_t^{-\gamma/2}\to \infty $ as $ t\to \infty $ and the drift term is of lower order compared to the collision operator. (Note that this is an a priori assumption that has to be check a posteriori, since $ \eta_t $ depends on $ f $.) This situation suggests that $ f $ remains close and converges to equilibrium. We hence use the following ansatz, which is the Hilbert-type expansion introduced in \cite{VelazqJames2019LongAsympHomoen},
\begin{align*}
		f_t(v)= \mu(v) + h_t^{(1)}(v) + \cdots + h^{(k)}_t(v) + \cdots.
\end{align*}
We assume as $ t\to \infty $ 
\begin{align*}
		h^{(1)}\ll \mu, \quad h^{(k+1)} \ll h^{(k)}, \quad k\in\N,
\end{align*}
and we can decompose
\begin{align}\label{eq:DefAlphaDecompInversTemp}
		\alpha_t  = \dfrac{1}{3}\int v\cdot L_tv \, f_t(v)\,dv = \dfrac{\trace L_t}{3} + \alpha_t^{(1)} + \cdots, \quad \alpha_t^{(k)} := \dfrac{1}{3}\int_{\R^3} v\cdot L_tv \, h^{(k)}_t(v)\,dv.
\end{align}
As a consequence we observe $ \alpha_t^{(k+1)}\ll \alpha_t^{(k)} $ as $ t\to \infty $. We plug the above ansatz into \eqref{eq:IndtrodModelHomoengBE} and collect terms of equal order (one has to take into account that $ \eta_t = \rho_t\, \beta_t^{-\gamma/2}\to \infty $ as $ t\to\infty $). The first order term $ h^{(1)} $ satisfies 
\begin{align*}
		0 = \div\left( \left( L_t - \dfrac{1}{3}\trace L_t  \, I\right) v\, \mu \right) + \eta_t \Lin h_t^{(1)}.
\end{align*}
One can show that the first term on the right-hand side is orthogonal to $ \ker \Lin $ w.r.t. the scalar product in $ L^2(\mu^{-1/2}) $. Hence, we can invert $ \Lin $ and obtain
\begin{align}\label{eq:HilbertExpFirstOrder}
		h_t^{(1)} = -\dfrac{1}{\eta_t} \Lin^{-1} \left[ v\cdot  A_t v \, \mu \right], \quad A_t := L_t - \dfrac{1}{3}\trace L_t  \, I.
\end{align}
Furthermore, we have
\begin{align}\label{eq:IntrodFirstOrderInversTemp}
		\alpha_t^{(1)} = -\dfrac{a_t}{3\eta_t}, \quad a_t:= \dualbra{v\cdot A_tv \, \mu}{ \Lin^{-1} \left[ v\cdot A_t v \, \mu\right] }_{L^2(\mu^{-1/2})}.
\end{align}
Note that $ a_t>0 $ for $ A_t\neq0 $, since $ \Lin $ is a positive operator on $ (\ker\Lin)^\perp $. We also observe that $ h_t^{(1)}= \mathcal{O}(1/\eta_t) $, recalling $ \eta_t\to \infty $, so that $ h_t^{(1)} \ll \mu $ as $ t\to \infty $. Similarly, one can formally solve the equations for $ h^{(k)} $ and conclude $ h^{(k)} = \mathcal{O}(1/\eta_t^k) $. In each equation the term $ \alpha_t^{(k)} $ allows to invert the operator $ \Lin $ on $ (\ker\Lin)^\perp $ as for $ k=0 $ above. Hence, the functions $ \alpha_t^{(k)} $ can be interpreted as Lagrange multipliers.

Finally, we need to show a posteriori that $ \eta_t = \rho_t\beta_t^{-\gamma/2} $ as $ t\to \infty $. As was observed in \cite{VelazqJames2019LongAsympHomoen} this is possible for $ L_t $ given by \eqref{eq:SimpleShear}, \eqref{eq:SimpShearDecPlanarDil} or \eqref{eq:CombOrthShear}. Let us consider here the case of simple shear \eqref{eq:SimpleShear}, so that $ L_t=L_0 $ is constant in time and $ \rho_t\equiv 1 $, due to $ \trace L_t =0 $. We use \eqref{eq:IntrodInvTempEq} and \eqref{eq:IntrodFirstOrderInversTemp} to obtain for $ \eta_t=\beta_t^{-\gamma/2} $
\begin{align*}
		\eta_t ' = \dfrac{\gamma a_0}{3} + \mathcal{O}(1/\eta_t).
\end{align*}
Here, we used that $ \alpha_t^{(k)}= \mathcal{O}(1/\eta_t^{k}) $ for the terms with $ k\geq 2 $. Hence, we get 
\begin{align*}
		\eta_t = \dfrac{\gamma a_0}{3} t+ o(t),
\end{align*}
i.e. $ \eta_t \to \infty $ as $ t\to \infty $, and thus
\begin{align*}
		\beta_t = \dfrac{\gamma a_0}{3} t^{-2/\gamma}(1+o(1)).
\end{align*}
The temperature $ T_t $ goes to infinity like $ t^{2/\gamma} $. 

Below we give a rigorous proof of the above longtime behavior $ f_t\to \mu $ and exact asymptotic formulas for the temperature in all the cases \eqref{eq:SimpleShear}, \eqref{eq:SimpShearDecPlanarDil} and \eqref{eq:CombOrthShear}. There are two crucial assumptions that we use. We assume that initially $ f_0 $ is close enough to a Maxwellian and that the initial temperature $ T_0=\beta_0^{-1} $ is sufficiently large. The latter condition ensures that the collision operator is dominant for all times, due to the term $ \rho_t\beta^{-\gamma/2}_t $.

\paragraph{Physical interpretation.} Let us briefly comment on the physical picture of the formal asymptotic study. To clarify the effect of the drift term $ L_tw\cdot\nabla_w g $ in \eqref{eq:homenergBE} we consider the flow $ P_t\in\R^{3\times 3} $ of $ -L_t $, that is
\begin{align*}
	P_t'=-L_tP_t, \quad P_0=I, \quad \det P_t = \exp\left( -\int_0^t\trace L_s \, ds \right).
\end{align*} 
The sign of $ \trace L_t $ determines whether we have expansion or dilatation in velocity space. In the case of homoenergetic solutions, we always have $ \trace L_t\geq0 $ to highest order. Thus, the dilatation would lead to a decrease of velocities and hence the temperature. But there is also the shearing effect due to the trace-free part $ A_t $ of $ L_t $. If the temperature is already high enough ($ \beta_0 $ small), most particles have very large velocities and the shear leads to an increase of them.

In total we have two competing effects. Both are present in the zeroth and first order terms in the formula for the inverse temperature \eqref{eq:IntrodInvTempEq}. More precisely, we have with \eqref{eq:DefAlphaDecompInversTemp} and \eqref{eq:HilbertExpFirstOrder}
\begin{align*}
	\left( \beta_t^{-\gamma/2} \right)' = -\gamma\,  \alpha_t \,  \beta_t^{-\gamma/2} \approx -\gamma \left( \alpha_t^{(0)}+\alpha_t^{(1)} \right)  \beta_t^{-\gamma/2} =   -\dfrac{\gamma\,  \trace L_t}{3} \beta_t^{-\gamma/2} + \dfrac{\gamma a_t}{3 \rho_t}.
\end{align*}
The shear is present through the term $ a_t $ given in \eqref{eq:IntrodFirstOrderInversTemp}, which only depends on $ A_t $. The dilatation is visible through $ \trace L_t\geq0 $. In the case that $ L_t $ is given by \eqref{eq:SimpleShear}, \eqref{eq:SimpShearDecPlanarDil} or \eqref{eq:CombOrthShear}, the shear is always more dominant than the dilatation. As a result we have $ \beta_t^{-1}=T_t\to\infty $ as $ t\to \infty $.

\subsection{Main results}
\paragraph{Notation.} Let us define the weighted $ L^p $-spaces $ L^p(w) $ with a positive weight function $ w:\R^3\rightarrow (0,\infty) $ with norm
\begin{align*}
		\norm[L^p(w)]{f}:=\norm[L^p]{w\,f}.
\end{align*}
For $ p=2 $ the scalar product is given by
\begin{align*}
		\dualbra{f}{g}_{L^2(w)}=\int_{\R^3}f(v)\, g(v)\, w(v)^2\, dv.
\end{align*}
In the case $ w(v)=\left\langle v  \right\rangle^m  $, $ m\geq0 $, we abbreviate them by $ L^p_m $, where $ \left\langle  v \right\rangle:=\sqrt{1+|v|^2} $. We also use the corresponding weighted Sobolev spaces $ W^{k,p}_m(\R^3) $, $ k\in\N_0 $, $ p\geq1 $. More precisely, $ f\in W^{k,p}_m $ if and only if the weak derivatives of order less or equal $ k\in\N_0 $ exist $ \partial^\alpha f $, for any multi-index $ \alpha\in\N^3_0 $, $ |\alpha|=\alpha_1+\alpha_2+ \alpha_3\leq k $ and
\begin{align*}
	\norm[W^{k,p}_m]{f}^p:=\sum_{|\alpha|\leq k} \norm[L^p_m]{\partial^\alpha f}^p
\end{align*}
is finite. Let us also define the norm of the homogeneous fractional Sobolev space $ \dot{H}^s $ for $ s\geq 0 $ by
\begin{align*}
	\norm[\dot{H}^s]{f}^2:= \int_{\R^3} |\xi|^{2s}|\mathscr{F}[f](\xi)|^2\, d\xi,
\end{align*}
where $ \mathscr{F}[f] $ is the Fourier transform. The norm of the inhomogeneous Sobolev space $ H^s $ for $ s\geq0 $ is defined by
\begin{align*}
	\norm[H^s]{f}^2= \norm[L^2]{f}^2+\norm[\dot{H}^s]{f}^2.
\end{align*}
Recall that $ W^{k,2}=H^k $ for $ k\in\N $ and we use both definitions interchangeably. The corresponding spaces with weights $ \left\langle v \right\rangle^m  $ are denoted by $ \dot{H}^s_m $ respectively $ H^s_m $ and the norms are given by
\begin{align*}
	\norm[\dot{H}^s_m]{f} = \norm[\dot{H}^s]{\left\langle \cdot \right\rangle^m\, f}, \quad \norm[H^s_m]{f} = \norm[H^s]{\left\langle \cdot \right\rangle^m\, f}.
\end{align*}
Finally, we write $ A\lesssim B $ resp. $ A\gtrsim B $, if there is a positive constant $ C>0 $ with $ A\leq CB $ resp. $ CA\geq B $. We write $ A\approx B $ if both $ A\lesssim B $ and $ A\gtrsim B $. 

\paragraph{Assumptions on the kernel.} We make the following assumptions.
\begin{itemize}
	\item (A-1) The collision kernel has the product form
	\begin{align*}
		B(v-v_*,\sigma)= b(n\cdot \sigma)|v-v_*|^\gamma.
	\end{align*}
	\item (A-2) The function $ b:[-1,1)\rightarrow[0,\infty) $ is locally smooth and has the angular singularity
	\begin{align}\label{eq:AssumptCorssSectionNonCutoff}
		\sin\theta \, b(\cos\theta) \, \theta^{1+2s}\to K_b>0, \quad \text{as } \theta\to 0
	\end{align}
	for some $ s\in(0,1/2) $ and $ K_b>0 $.
	\item (A-3) The parameter $ \gamma $ satisfies $ \gamma\in(0,1) $.
\end{itemize}
In particular, this implies
\begin{align}\label{eq:AssumptCrossSection}
	\Lambda =\int_0^{\pi} \sin\theta \, b(\cos\theta) \, \theta d\theta <\infty.
\end{align}
These assumptions cover inverse power law interactions with $ q>5 $, cf. \eqref{eq:CollisinKernel} and \eqref{eq:angularSing}. Finally, we also assume without loss of generality that $ b(\cos\theta) $ is supported on $ [0,\pi/2] $ by using the symmetrization
\begin{align*}
	b(n\cdot \sigma)\ind_{\set{n\cdot \sigma \geq 0}}+ b(-n\cdot \sigma)\ind_{\set{n\cdot \sigma \geq 0}}.
\end{align*}
This does not change the collision operator $ Q(f,f) $, since $ f(v')f(v'_*) $ is invariant under the change of variables $ \sigma \mapsto -\sigma $.

\paragraph{Results for homoenergetic solutions with collision-dominated behavior.} We consider solutions $ g $ to \eqref{eq:homenergBE} with initial mass $ \rho_0=1 $, momentum $ V_0\in \R^3 $ and inverse temperature $ \beta_0>0 $. They are related to solutions $ f_t(v)=g_t(v\beta_t^{-1/2}+V_t)\beta_t^{-3/2}\rho_t^{-1} $ to the equations \eqref{eq:IndtrodModelHomoengBE}. Here, we used
\begin{align}\label{eq:MassMomentum}
		\rho_t = \exp\left( -\int_0^t\trace L_s \, ds \right), \quad V_t = P_tV_0, \quad \dfrac{1}{\beta_t} = \dfrac{1}{3\rho_t} \int_{\R^3}\left| w-V _t\right|^2g_t(w)\, dw
\end{align} 
with $ t\mapsto P_t\in \R^{3\times3} $ satisfying $ P_t'=-L_tP_t $, $ P_0=I $. The well-posedness and regularity theory of these equations is discussed in Section \ref{sec:WellPosedness}, see Proposition \ref{pro:Wellposedness} and Proposition \ref{pro:Regularity}. For our main results we use the following weighted Sobolev space $ \mathcal{H}^1_p $ with norm
\begin{align*}
	\norm[\mathcal{H}^1_p]{f}^2:= \norm[L^2_p]{f}^2 + \sum_{|\alpha|= 1}\norm[L^2_{p-2s}]{\partial^\alpha f}^2,
\end{align*}
where $ s\in(0,1/2) $ is given in \eqref{eq:AssumptCorssSectionNonCutoff}.
\begin{thm}\label{thm:HomoenergticSol1}
	Consider equation \eqref{eq:homenergBE} with matrix $ L_t =L_0(I+tL_0)^{-1}  $ having the asymptotic form \eqref{eq:SimpleShear}, \eqref{eq:SimpShearDecPlanarDil} or \eqref{eq:CombOrthShear}. Let $ p_0> 4+4s +3/2 $ be arbitrary and $ g_0\in \mathcal{H}^1_{p_0} $. Consider the unique solution $ g $ to \eqref{eq:homenergBE}. Define with \eqref{eq:MassMomentum}
	\begin{align}\label{eq:HilbertExpFirstOrder2}
		f_t(v):=g_t(v\beta_t^{-1/2}+V_t)\beta_t^{-3/2}\rho_t^{-1}, \quad  \bar{\mu}_t := \dfrac{1}{\rho_t\beta_t^{-\gamma/2}}\Lin^{-1}\left[ -v\cdot A_t v \mu \right], \quad A_t:= L_t-\dfrac{\trace L_t}{3}\, I
	\end{align} 
	and $ h_t(v) := f_t(v)-\mu(v)-\bar{\mu}_t(v) $. 
	
	There are $ \varepsilon_0\in(0,1) $ sufficiently small and a constant $ C'>0 $, depending only on $ p_0 $, $ L_t $ and the collision kernel $ B $, such that: If $ \norm[\mathcal{H}^1_{p_0}]{h_0}=\varepsilon\leq \varepsilon_0 $ and $ \beta_0\leq \varepsilon_0 $, we have 
	\begin{align}\label{eq:ConvergenceToEqui}
		\norm[\mathcal{H}^1_{p_0}]{f_t-\mu}\to 0
	\end{align}
	as $ t\to \infty $. Furthermore, the inverse temperature has the following asymptotics in each case.
	
	\begin{enumerate}[(i)]
		\item Simple shear, $ L_t $ given by \eqref{eq:SimpleShear}: We have 
		\begin{align}\label{eq:AsymptoticsInversTempSimpleShear}
			\lim_{t\to\infty} \dfrac{\beta_t^{-\gamma/2}}{t}=\dfrac{\gamma \bar{a}}{3}
		\end{align}
		with the constant $ \bar{a}>0 $ given by
		\begin{align}\label{eq:AsymptoticsSolCorollary1Constant}
			\bar{a}:=\dualbra{v\cdot A^0v\, \mu}{\Lin^{-1}\left[ v\cdot A^0v \, \mu \right] }_{L^2(\mu^{-1/2})} , \quad A^0 := \left( \begin{array}{ccc}
				0 & K & 0 \\ 
				0 & 0 & 0 \\ 
				0 & 0 & 0
			\end{array} \right).
		\end{align}
		
		\item Simple shear with decaying planar dilatation/shear, $ L_t $ given by \eqref{eq:SimpShearDecPlanarDil}: We have
		\begin{align}\label{eq:AsymptoticsInversTempSimpleShearDecaying}
			\lim_{t\to\infty} \dfrac{\beta_t^{-\gamma/2}}{t^2} = \dfrac{\gamma \bar{a}}{\gamma+6}\exp\left( \int_0^\infty r_s \, ds \right),
		\end{align}
		where we defined
		\begin{align}\label{eq:AsymptoticsSolCorollary2Constant}
			r_t:= \trace L_t - \dfrac{1}{1+t} = \mathcal{O}\left( \dfrac{1}{(1+t)^2} \right), \quad t\to \infty,
		\end{align}
		and the constant $ \bar{a}>0 $ is given by
		\begin{align}\label{eq:AsymptoticsSolCorollary22Constant}
			\bar{a} := \dualbra{v\cdot A^0 v \, \mu}{\Lin^{-1}\left[ v\cdot A^0 v \, \mu \right] }_{L^2(\mu^{-1/2})}, \quad A^0 := \left( \begin{array}{ccc}
				0 & K_2 & 0 \\ 
				0 & 0 & 0 \\ 
				0 & 0 & 0
			\end{array} \right).
		\end{align}
		
		\item Combined orthogonal shear, $ L_t $ given by \eqref{eq:CombOrthShear}: It holds
		\begin{align}\label{eq:AsymptoticsInversTempCombOrthogShear}
			\lim_{t\to\infty} \dfrac{\beta_t^{-\gamma/2}}{t^3} = \dfrac{\gamma \bar{a}}{9},
		\end{align}
		where the constant $ \bar{a}>0 $ is defined by
		\begin{align}\label{eq:AsymptoticsSolCorollary3Constant}
			\bar{a}:= \dualbra{v\cdot A^0v \, \mu}{\Lin^{-1}\left[ v\cdot A^0 v \mu \right]}_{L^2(\mu^{-1/2})}, \quad A^0:= \left( \begin{array}{ccc}
				0 & 0 & -K_1K_3 \\ 
				0 & 0 & 0 \\ 
				0 & 0 & 0
			\end{array} \right).
		\end{align}
	\end{enumerate}
\end{thm}
In fact we can give a more quantitative statement than \eqref{eq:ConvergenceToEqui}, which also verifies the formal Hilbert-type expansion discussed previously.

\begin{thm}\label{thm:HomoenergticSol2}
	Under the assumptions of Theorem \ref{thm:HomoenergticSol1} the following statements hold in addition.
	\begin{enumerate}[(i)]
		\item Simple shear, $ L_t $ given by \eqref{eq:SimpleShear}: We have for all $ t\geq0 $
		\begin{align}\label{eq:FirstSecondOrderHilbertExpDecay1}
			\norm[\mathcal{H}^1_{p_0}]{h_t}\leq C'\left( \dfrac{\varepsilon}{(1+t)^2} +\dfrac{1}{\zeta_t^2}\right), \quad \norm[H^k_p]{\dfrac{\bar{\mu}_t}{\sqrt{\mu}}} \leq \dfrac{C_{k,p}}{\zeta_t}
		\end{align}
		for all $ t\geq0 $, $ k,\, p\in\N $ and 
		\begin{align*}
			\zeta_t:=\beta_0^{-\gamma/2}+t.
		\end{align*}
	
		\item Simple shear with decaying planar dilatation/shear, $ L_t $ given by \eqref{eq:SimpShearDecPlanarDil}: For all $ t\geq0 $ we have
		\begin{align}\label{eq:FirstSecondOrderHilbertExpDecay2}
			\norm[\mathcal{H}^1_{p_0}]{h_t}\leq C'\left( \dfrac{\varepsilon}{(1+t)^2} +\dfrac{1}{\zeta_t^2}\right), \quad \norm[H^k_p]{\dfrac{\bar{\mu}_t}{\sqrt{\mu}}} \leq \dfrac{C_{k,p}}{\zeta_t} 
		\end{align}
		for all $ t\geq0 $, $ k,\, p\in\N $ and
		\begin{align*}
			\zeta_t:=\beta_0^{-\gamma/2}(1+t)^{-1-\gamma/3} + t.
		\end{align*}
		
		\item Combined orthogonal shear, $ L_t $ given by \eqref{eq:CombOrthShear}: For any $ t\geq0 $ it holds
		\begin{align}\label{eq:FirstSecondOrderHilbertExpDecay3}
			\norm[\mathcal{H}^1_{p_0}]{h_t}\leq C'\left( \dfrac{\varepsilon}{(1+t)^4}+\dfrac{1}{\zeta_t^2}\right), \quad \norm[H^k_p]{\dfrac{\bar{\mu}_t}{\sqrt{\mu}}} \leq \dfrac{C_{k,p}}{\zeta_t} 
		\end{align}
		for all $ t\geq0 $, $ k,\, p\in\N $ and
		\begin{align*}
			\zeta_t:=\beta_0^{-\gamma/2}(1+t)^{-1}+t^2.
		\end{align*}
	\end{enumerate}
\end{thm}
\begin{rem} Let us give a few comments.
	\begin{enumerate}[(i)]
		\item Note that the estimates for $ h_t $ in \eqref{eq:FirstSecondOrderHilbertExpDecay1} respectively \eqref{eq:FirstSecondOrderHilbertExpDecay2} contain a quadratic decay whereas in \eqref{eq:FirstSecondOrderHilbertExpDecay3} the decay is of order four. Similarly, for $ \bar{\mu}_t $ the decay is faster in \eqref{eq:FirstSecondOrderHilbertExpDecay3}. The reason is that in this case the matrix $ L_t $ is given by \eqref{eq:CombOrthShear}, for which the shear is growing linearly in time. Since the shear is the driving mechanism for the temperature to grow, the process is accelerated.
		
		\item As we will see in Section \ref{sec:CollDomBehav}, the equation satisfied by $ h $ contains a source term, see \eqref{eq:ErrorEq}. This term leads to a decay of order $ 1/\zeta_t^2 $. The other terms contain $ h $ and imply a decay of order $ \varepsilon/(1+t)^2 $. This yields the particular form of the above estimates. Note that for $ t $ of order $ \beta_0^{-\gamma/2} $ the term $ 1/\zeta_t^2 $ is larger than $ \varepsilon/(1+t)^2 $ and $ \norm{h_t}\lesssim 1/(1+t)^2 $. For times of order one, it depends on the relative size of $ \varepsilon, \, \beta_0^{\gamma/2} $ to see which term is larger.
		
		\item Our estimates on the perturbation $ h $, in particular the use of the norm $ \norm[\mathcal{H}^1_{p_0}]{\cdot} $, rely on results in \cite{HerauTononTristani2020RegEstCauchyInhomBoltzmann}, where polynomially decaying solutions to the inhomogeneous Boltzmann equation close to equilibrium have been constructed. However, we combine them with the stability of $ \Lin $ in $ L^1_m $, for some $ m>2 $, proved in \cite{Tristani2014ExponentialConvergenceEquilHomogBoltz}.
		
		\item In the main part of the paper, we consider the non-cutoff case. In Section \ref{sec:Cutoff} we discuss a variant of the above theorem in the cutoff case. The proof follows the main arguments in Section~\ref{sec:CollDomBehav}, but uses less technical estimates on the collision operator.
	\end{enumerate}
\end{rem}
The paper is organized in the following way. In Section \ref{sec:WellPosedness}, we show well-posedness and regularity for equation \eqref{eq:homenergBE} with general initial data. Then, we prove the more general Theorem \ref{thm:MainTheorem} based on regularity estimates on the level of the linearization $ h $ in Section \ref{sec:CollDomBehav}. In Section \ref{sec:Application}, we conclude Theorem \ref{thm:HomoenergticSol1} and Theorem \ref{thm:HomoenergticSol2} as an application of Theorem \ref{thm:MainTheorem}. Finally, in Section \ref{sec:Cutoff} we discuss how the above theorems can be proved in the cutoff case.

\section{Well-posedness and regularity for homoenergetic solutions}
\label{sec:WellPosedness}
In this section, we study existence, uniqueness and regularity of solutions $ f $ to
\begin{align}\label{eq:homenergBEWellPosReg}
	\partial_tf = L_tv\cdot \nabla f_t+Q(f,f).
\end{align} 
The matrix $ L_t $ is given, not necessarily satisfying \eqref{eq:homenergBE}.

Let us introduce the notion of weak solutions to \eqref{eq:homenergBEWellPosReg}, which is reminiscent of weak solutions to the homogeneous Boltzmann equation. Recall that the entropy $ H(f) $ of some function $ f\geq0 $ is given by
\begin{align*}
	H(f) = \int_{\R^3} f(v) \log f(v)\, dv.
\end{align*}

\begin{defi}
	Let $ f_0\in L^1_2 $ with $ H(f_0)<\infty $. We say that $ f\in L^\infty_{\mathrm{loc}}([0,\infty);L^1_2) $, $ f\geq0 $ is a weak solution to \eqref{eq:homenergBEWellPosReg} if for all $ T\geq0 $ and all test functions $ \varphi\in C^1_b([0,T]\times \R^3) $ it holds
	\begin{align}\label{eq:WeakForm}
		\begin{split}
			&\int_{\R^3}f_T(v)\varphi_T(v)\, dv -\int_{\R^3}f_0(v)\varphi_0(v)\, dv - \int_0^T \int_{\R^3}f_s\, \partial_s\varphi_s \, dsdv
			\\
			&= - \int_0^T\int_{\R^3} \div \left( L_sv\, \varphi_s \right) \, f_s(v)\, dvds +\int_0^T \dualbra{Q(f_s,f_s)}{\varphi_s}\, ds.
		\end{split}
	\end{align}
	Furthermore, $ f $ satisfies for all $ t\geq0 $
	\begin{align}\label{eq:Entropy}
		H(f_t) &\leq H(f_0)\exp\left( -\int_0^t\trace L_s\, ds\right).
	\end{align}
\end{defi}
Here, we interpret
\begin{align*}
\dualbra{Q(f_t,f_t)}{\varphi_t} = \int_{\R^3}\int_{\R^3}\int_{S^2} |v-v_*|^\gamma b(n\cdot \sigma)\, f_t(v)f_t(v_*)\, (\varphi_t(v')-\varphi_t(v)) \,d\sigma dv_*dv.
\end{align*}
This is motivated by testing $ Q(f_t,f_t) $ with $ \varphi $ and applying the pre-post-collisional change of variables. Note that
\begin{align*}
\int_{S^2}b(n\cdot \sigma)(\varphi(v')-\varphi(v))\, d\sigma\leq \Lambda\left( \sup_{|\xi|\leq \sqrt{|v|^2+|v_*|^2}}|\nabla \varphi(\xi)| \right) \, |v-v_*|,
\end{align*}
which follows from $ |v'-v|= |v-v_*|\sin(\theta/2) $ and our assumptions on $ b $, see \eqref{eq:AssumptCrossSection}. Hence, we have
\begin{align*}
|\dualbra{Q(f_t,f_t)}{\varphi}| \leq C\Lambda\, \norm[\infty]{\nabla\varphi}\norm[L^1_{1+\gamma}]{f_t}^2
\end{align*}
and the weak formulation is well-defined due to $ 1+\gamma\leq 2 $. In order to motivate \eqref{eq:Entropy} it is convenient to integrate \eqref{eq:homenergBEWellPosReg} by characteristics yielding
\begin{align*}
	\partial_t \left[ f_t(P_{0,t}v) \right] = Q(f_t,f_t)(P_{0,t}v),
\end{align*} 
where $ P_{0,t}\in \R^{3\times3} $ satisfies $ P_{0,t}'=-L_tP_{0,t} $, $ P_{0,0}=I $. We formally calculate
\begin{align*}
	\frac{d}{dt}\left[ \det(P_{0,t})^{-1}H(f_t)\right] = \int_{\R^3}Q(f_t,f_t)(P_{0,t}v)\, \log f_t(P_{0,t}v)\, dv \leq 0,
\end{align*}
which implies \eqref{eq:Entropy}.

\begin{pro}\label{pro:Wellposedness}
	Consider \eqref{eq:homenergBEWellPosReg} with $ L\in L^\infty_{\mathrm{loc}}([0,\infty);\R^{3\times3})  $. Then the following statements hold.
	\begin{enumerate}[(i)]
		\item Let $ p>2 $ be arbitrary. For any $ f_0\in L^1_{p} $ with $ H(f_0)<\infty $ there is a weak solution $ f \in L^\infty_{\mathrm{loc}}([0,\infty); L^1_{p}) $ to \eqref{eq:homenergBEWellPosReg}. Furthermore, for any $ t_0>0 $ and any $ q\in\N $, $ T\geq t_0 $
		\begin{align*}
			\sup_{t\in[t_0,T]}\norm[L^1_q]{f_t}  \leq C.
		\end{align*}
		Here, $ C $ depends on $ t_0,\, q,\, \sup_{t\in[0,T]}\norm{L_t}  $ and $ T $.
		
		\item Let $ q\geq 2 $. Then, there is at most one weak solution $ f \in L^\infty_{\mathrm{loc}}([0,\infty); W^{1,1}_{q+1+\gamma}) $ to \eqref{eq:ModelEq}.
	\end{enumerate}
\end{pro}
To prove this proposition we recall the following version of the Povzner estimate proved in \cite[Sect. 2]{MischlerWennberg1999HomogBE}. As was noticed e.g. in \cite[Appendix]{Villani1998NewClassWeakSol}, their calculation also works in the non-cutoff case.
\begin{lem}\label{lem:Povzner}
	Let $ \varphi(v)=|v|^{2+\delta} $ for $ \delta>0 $. Then we have the following decomposition
	\begin{align*}
	\int_{S^2}b(n\cdot \sigma )\left\lbrace \varphi'_*+\varphi'-\varphi_*-\varphi \right\rbrace d\sigma = G(v,v_*)-H(v,v_*)
	\end{align*}
	with $ G,H $ satisfying
	\begin{align}
	\begin{split}\label{eq:PovznerDecomposition}
	G(v,v_*) &\leq C\Lambda(|v||v_*|)^{1+\delta/2},
	\\
	H(v,v_*) &\geq c\Lambda(|v|^{2+\delta}+|v_*|^{2+\delta}) \left( 1-\ind_{\set{|v|/2< |v_*|< 2|v|}} \right).
	\end{split}
	\end{align}
	for some $ c,C>0 $ depending on $ \delta $.
\end{lem}
\begin{proof}[Proof of Proposition \ref{pro:Wellposedness}]
	\textit{(i)} To prove existence we first consider the case of angular cutoff, as in the analysis of the homogeneous Boltzmann equation, see e.g. \cite{Villani1998NewClassWeakSol}.
	
	\textit{Step 1:} Let us consider a cutoff collision kernel, i.e. for $ n\in \N $, $ n\in \N $, $ B_n:=(|v-v_*|\wedge n)^\gamma \, [b(\cos\theta)\wedge n] $. The corresponding collision operator is denoted by $ Q_n $. Solutions $ f^n $ to the corresponding problem with finite entropy were constructed by Cercignani in \cite{Cercignani1989ExistenceHomoenergetic}. The main idea was to study the problem by integrating via characteristics and to adapt arguments in \cite{Arkeryd72OnBoltzmannEq}. Furthermore, the mass, momentum and energy/temperature satisfy the a priori estimates in \eqref{eq:MacroscopicQuantities}. In particular, $ \norm[L^1_2]{f^n_t} $ is bounded uniformly in $ n\in\N $, locally in time.
	
	If $ f_0\in L^1_p $ then we have $ f^n\in L^\infty_{\mathrm{loc}}([0,\infty);L^1_p) $. We use the above Povzner estimates to obtain bounds in $ L^1_p $ uniformly in $ n\in\N $, $ p>2 $. We have with Lemma \ref{lem:Povzner} for $ M_p^n(t):=\int_{\R^3}|v|^pf_t^n\, dv $
	\begin{align*}
		\dfrac{d}{dt}M_p^n &\leq C(p,L)M_p^n+C\Lambda_n M_{\gamma+p/2}^nM_{p/2}^n
		\\
		&-c\Lambda_n\int_{\R^3}\int_{\R^3}(|v-v_*|\wedge n)^\gamma |v|^p\, f_t^n(v)f_t^n(v_*)\, dv_*dv.
	\end{align*}
	Here, we used for the last term in $ H(v,v_*) $ in \eqref{eq:PovznerDecomposition}, $ |v|^{p}\ind_{\set{|v|/2< |v_*|< 2|v|}} \lesssim (|v||v_*|)^{p/2} $. We apply
	\begin{align*}
		(|v-v_*|\wedge n)^\gamma \geq \dfrac{1}{4}(|v|\wedge n)^\gamma-|v_*|^\gamma
	\end{align*}
	to get with $ M_\gamma\leq \norm[L^1_2]{f^n_t} $
	\begin{align*}
		\dfrac{d}{dt}M_p^n\leq CM_p^n + C\Lambda_n M_{\gamma+p/2}^nM_{p/2}^n+c\Lambda_nM_p^n-\dfrac{c\Lambda_n}{4}\int_{\R^3}(|v|\wedge n)^\gamma |v|^p\, f_t^n(v)dv.
	\end{align*}
	If $ p\leq 4 $ we have $ M_{p/2}^n\leq \norm[L^1_2]{f^n_t} $ and we can use ($ \gamma+p/2<p $)
	\begin{align*}
		M_{\gamma+p/2}^n \leq C_\varepsilon M_0^n+\varepsilon \int_{\R^3}(|v|\wedge n)^\gamma |v|^p\, f_t^n(v)dv
	\end{align*}
	for all $ \varepsilon>0 $. Hence, a Gronwall argument applies and $ M_p^n $ is bounded locally in time. This is uniform in $ n\in\N $, since $ \Lambda_n\leq \Lambda $. One can see, using the weak formulation, that $ t\mapsto\int f^n_t\varphi\, dv $ is Lipschitz, uniformly in $ n\in\N $, for any test function. The entropy bound \eqref{eq:Entropy} yields then weak $ L^1 $ compactness by the Dunford-Pettis theorem, $ f^{n_k}\rightharpoonup f $ for a limit $ f\in L^\infty_{\mathrm{loc}}(L^1_p) $ and a subsequence $ n_k\to \infty $. Furthermore, we can pass to the limit in the definition of the weak formulation \eqref{eq:WeakForm}, since $ p>2 $. In particular, $ f $ satisfies \eqref{eq:MacroscopicQuantities}.
	
	In the case $ p> 4 $, one can use the above reasoning inductively, so that the term $ M_{p/2}^n $ can be bounded by the previous inductive step.
	
	\textit{Step 2:} We now prove the gain of moments using the Povzner estimates similarly to the homogeneous Boltzmann equation. One argues again inductively. We obtain the estimate by testing with $ |v|^p $
	\begin{align*}
		\dfrac{d}{dt}M_p \leq CM_p + C\Lambda M_{\gamma+p/2}M_{p/2} - c\Lambda M_{p+\gamma}.
	\end{align*}
	From a previous inductive step we know that $ M_{p/2} $ is bounded locally in time. We can use $ M_p+M_{\gamma+p/2} \leq C_\varepsilon M_0+\varepsilon M_{p+\gamma} $ to get
	\begin{align*}
		\dfrac{d}{dt}M_p \leq C_\varepsilon M_0- \dfrac{c\Lambda}{2} M_{p+\gamma}.
	\end{align*}
	An integration yields
	\begin{align*}
		M_p(T)+\dfrac{c\Lambda}{2} \int_0^TM_{p+\gamma}(s)\, ds \leq M_p(0)+C(T).
	\end{align*} 
	Hence, for any $ t_1>0 $ we can find $ t_0\in(0,t_1) $ such that $ M_{p+\gamma}(t_0) $ is finite. Thus, the solution gained a moment of order $ \gamma>0 $. One can then argue with $ p+\gamma $ instead of $ p $ and starting from the time $ t_0 $. However, the preceding argument was formal, but can be made rigorous when using a cutoff of $ |v|^p $ as a test function.
	
	\textit{(ii)} We prove uniqueness following the arguments in \cite[Theorem 1, Proposition 1]{DesvillettesMouhot2009StabilityUniquenessHomogBoltzmann}, where the homogeneous Boltzmann equation was considered. 
	
	Let $ f,\,\tilde{f} $ be two weak solutions to \eqref{eq:homenergBEWellPosReg}. Due to $ f,\, \tilde{f}\in L^\infty_{\mathrm{loc}}([0,\infty); W^{1,1}_{q+1+\gamma}) $ the right hand side in \eqref{eq:homenergBEWellPosReg} is in $ L^\infty_{\mathrm{loc}}([0,\infty); L^1_{q}) $, see Lemma \ref{lem:L1EstimateCollOp} for the estimate of the collision operator. Thus $ t\mapsto f_t,\, \tilde{f}_t\in L^1_q $ are Lipschitz. This allows to make the following reasoning rigorous. As in \cite{DesvillettesMouhot2009StabilityUniquenessHomogBoltzmann} set $ D=f-\tilde{f} $ and $ S=f+\tilde{f} $ to get the equation
	\begin{align*}
		\partial_t D = L_tv \cdot \nabla D + \dfrac{1}{2}\left( Q(S,D)+ Q(D,S) \right) .
	\end{align*}  
	We then obtain (by formally testing this equation with $ \text{sgn}(D)\left\langle v \right\rangle^q $ and integrating in time)
	\begin{align*}
		\norm[L^1_q]{D_T} &= \int_0^T\int_{\R^3}\left[ -\div \left( L_tv\left\langle v \right\rangle^q \right)\, |D_t| + \dfrac{1}{2}\left( Q(S_t,D_t)+Q(D_t,S_t)\right) \text{sgn}D_t\left\langle v \right\rangle^q \right] \, dvdt
		\\
		&\leq C_q\norm[\infty]{L}\int_0^T\norm[L^1_q]{D_t}\, dt+\int_0^T\int_{\R^3}\dfrac{1}{2}\left( Q(S_t,D_t)+Q(D_t,S_t)\right) \text{sgn}D_t\left\langle v \right\rangle^q \, dvdt.
	\end{align*}
	We use the same estimates as in \cite{DesvillettesMouhot2009StabilityUniquenessHomogBoltzmann} for the collision operator. The idea is to split into a cutoff and non-cutoff part $ Q=Q_{c,\varepsilon}+Q_{nc,\varepsilon} $ with respect to the angular part $ b=b_{c,\varepsilon}+b_{nc,\varepsilon} $ for a parameter $ \varepsilon>0 $. For the cutoff part, a variant of the Povzner estimate is used (see \cite[Lemma 1]{LuWennberg2002SolutionIncEnergy}) to get
	\begin{align*}
		&\dualbra{Q_{c,\varepsilon}(S_t,D_t)+Q_{c,\varepsilon}(D_t,S_t)}{\text{sgn}D_t\left\langle v \right\rangle^q} 
		\\
		&\leq \dualbra{Q_{c,\varepsilon}(S_t,|D_t|)+Q_{c,\varepsilon}(|D_t|,S_t)}{\left\langle v \right\rangle^q}+2\int_{\R^3}\int_{\R^3}\int_{S^2}|v-v_*|^\gamma b_{c,\varepsilon}(n\cdot \sigma)\, |D_*|\, S\left\langle v \right\rangle^q \, d\sigma dv_*dv
		\\
		&\leq C_\varepsilon\norm[L^1_q]{D_t}-K\norm[L^1_{q+\gamma}]{D_t}.
	\end{align*}
	For the non-cutoff part, we have
	\begin{align*}
		\dualbra{Q_{nc,\varepsilon}(S_t,D_t)}{\text{sgn}D_t\left\langle v \right\rangle^q} \leq \int_{\R^3}\int_{\R^3}\int_{S^2} |v-v_*|^\gamma b_{nc,\varepsilon}(n\cdot \sigma) \left( S'_*|D'|-S_*|D| \right) \left\langle v \right\rangle^q\, d\sigma dv_*dv.
	\end{align*}
	We can now us the pre-post collisional change of variables and the fact that $ |\left\langle v' \right\rangle^q-\left\langle v \right\rangle^q|\lesssim \sin(\theta/2)\left\langle v \right\rangle^q\left\langle v_* \right\rangle^q  $ to get
	\begin{align*}
		\dualbra{Q_{nc,\varepsilon}(S_t,D_t)}{\text{sgn}D_t\left\langle v \right\rangle^q} \leq c_\varepsilon \norm[L^1_{q+\gamma}]{S_t}\norm[L^1_{q+\gamma}]{D_t}.
	\end{align*}
	Finally, using Lemma \ref{lem:L1EstimateCollOp} we have
	\begin{align*}
		\norm[L^1_q]{Q_{nc,\varepsilon}(D_t,S_t)}\leq c_\varepsilon\norm[L^1_{q+\gamma}]{D_t}\norm[W^{1,1}_{q+\gamma+1}]{S_t}.
	\end{align*}
	Note that $ c_\varepsilon\to 0 $ as $ \varepsilon\to 0 $. For $ \varepsilon $ small enough we get in total 
	\begin{align*}
		\norm[L^1_q]{D_T} \leq (C_q\norm[\infty]{L}+C_\varepsilon)\int_0^T\norm[L^1_q]{D_t}\, dt
	\end{align*}
	for all $ T\geq0 $. Hence, we conclude $ D\equiv 0 $.
\end{proof}

Now we want to prove that the weak solutions constructed in Proposition \ref{pro:Wellposedness} are smooth for positive times. This is analogous to the homogeneous Boltzmann equation and follows from the singular behavior of the angular part of the collision kernel, see \cite{AlexandreDesvillettesVillaniWennberg2000EntropyDissip}. Here, we follow the treatment in \cite{AlexandreMorimoto2012SmoothingEffectWeakSolHomogBE} and show how to adapt the arguments for equation \eqref{eq:homenergBEWellPosReg} containing also a drift term. To this end, we use two lemmas proved in \cite{AlexandreMorimoto2012SmoothingEffectWeakSolHomogBE}. For the first one see \cite[Proposition 2.1, Proposition 3.8]{AlexandreMorimoto2012SmoothingEffectWeakSolHomogBE}.
\begin{lem}\label{lem:CoercivityUpperBound}
	The following two estimates hold.
	\begin{enumerate}[(i)]
		\item For any $ g\in L^1_2 $ with $ g\geq0 $, $ \norm[L^1_2]{g}+H(g)\leq E_0 $, $ \norm[L^1]{g}\geq D_0 $ for $ D_0,\, E_0>0 $ we have
		\begin{align*}
		-\dualbra{Q(g,f)}{f}_{L^2}\geq c_0\norm[H^s_{\gamma/2}]{f}^2-C\norm[L^2]{f}^2.
		\end{align*}
		Here, the constants $ c_0,\, C $ only depend on $ D_0,\, E_0 $. Recall that $ s\in(0,1/2) $ is given in \eqref{eq:AssumptCorssSectionNonCutoff}.
		
		\item For any $ r\in [2s-1,2s] $, $ \ell\in [0,\gamma+2s] $ we have
		\begin{align*}
		\left| \dualbra{Q(f,g)}{h}_{L^2} \right| \lesssim \norm[L^1_{\gamma+2s}]{f}\norm[H^r_{\gamma+2s-\ell}]{g}\norm[H^{2s-r}_{\ell}]{h}.
		\end{align*}
	\end{enumerate}
\end{lem}
Since we consider weak solutions an approximation procedure is necessary. As in \cite{AlexandreMorimoto2012SmoothingEffectWeakSolHomogBE} we define the mollifier in Fourier space via
\begin{align*}
M_\lambda^\delta (\xi) = \dfrac{\left\langle \xi \right\rangle^\lambda }{(1+\delta\left\langle \xi \right\rangle^{N_0} )}
\end{align*}
for $ \delta>0 $, $ \lambda, \, N_0\in \R $. This is a pseudo-differential symbol $ M_\lambda^\delta\in S^{\lambda-N_0}_{1,0} $ and we define accordingly
\begin{align*}
M_\lambda^\delta(D_v)f = \mathscr{F}^{-1}\left[ M_\lambda^\delta \mathscr{F}[f] \right], 
\end{align*}
where $ \mathscr{F}[f] $ denotes the Fourier transform of $ f $. We also abbreviate $ M_\lambda^\delta f $. The next lemma is a commutator estimate, see \cite[Theorem 3.6]{AlexandreMorimoto2012SmoothingEffectWeakSolHomogBE}.
\begin{lem}\label{lem:CommutatorErstimate}
	Let $ s'\in (0,s) $ and assume that $ \lambda,\, N_0 $ satisfies
	\begin{align}\label{eq:CommutatorCondition}
	5+\gamma \geq 2(N_0-\lambda).
	\end{align}
	\begin{enumerate}[(i)]
		\item If $ s'+\lambda<3/2 $ we have
		\begin{align*}
		\left| \dualbra{M_\lambda^\delta Q(f,g)-Q(f,M_\lambda^\delta g)}{h}_{L^2} \right|\lesssim \norm[L^1_\gamma]{f}\norm[H^{s'}_{\gamma/2}]{M_\lambda^\delta g} \norm[H^{s'}_{\gamma/2}]{h}.
		\end{align*}
		\item If $ s'+\lambda\geq3/2 $ we have
		\begin{align*}
		\left| \dualbra{M_\lambda^\delta Q(f,g)-Q(f,M_\lambda^\delta g)}{h}_{L^2} \right|\lesssim \left( \norm[L^1_\gamma]{f} + \norm[H^{(\lambda+s'-3)_+}]{f} \right) \norm[H^{s'}_{\gamma/2}]{M_\lambda^\delta g} \norm[H^{s'}_{\gamma/2}]{h}.
		\end{align*}
	\end{enumerate}
\end{lem}
\begin{pro}\label{pro:Regularity}
	Any weak solution $ f $ to \eqref{eq:homenergBEWellPosReg} with $ f\in L^\infty([0,T];L^1_p) $ for all $ p\geq0 $ satisfies for all $ k,\, p\geq0 $ and any $ t_0>0 $
	\begin{align*}
	f\in L^\infty([t_0,T];H^k_p).
	\end{align*}
\end{pro}
Note that due to Proposition \ref{pro:Wellposedness} (i) the above assumptions are satisfied for positive times.
\begin{proof}[Proof of Proposition \ref{pro:Regularity}]
	The proof is similar to the original one in \cite[Theorem 4.1, Theorem~5.1]{AlexandreMorimoto2012SmoothingEffectWeakSolHomogBE}.
	
	\textit{Step 1:} First we prove that $ f\in L^\infty([T_0,T];L^2_\ell) $ for all $ \ell\geq0 $ and some $ T_0\geq0 $ implies the claim. We do this by induction and indicate the induction step. Accordingly, let us assume w.l.o.g. that for some $ a\geq0 $ and any $ \ell\geq0 $ we have
	\begin{align*}
		f\in L^\infty([0,T];H^a_\ell).
	\end{align*}
	Choose $ T_1>0 $ arbitrary. We define $ \lambda(t):= Nt+a $ for $ N>0 $ with $ NT_1=(1-s) $ and $ N_0:=a+(5+\gamma)/2 $. In particular, \eqref{eq:CommutatorCondition} holds. For any $ t\in [0,T_1] $ we have
	\begin{align*}
		\lambda(t)-N_0-a\leq \lambda(T_1)-N_0= 1-s -(5+\gamma)/2 < -3/2.
	\end{align*}
	Hence, we have for all $ p\geq0 $
	\begin{align}\label{eq:Step1ProofApproxRegularity}
		M_{\lambda(t)}^\delta f_{t'}\in L^\infty ([0,T]\times [0,T]; H^{3/2}_p\cap L^\infty).
	\end{align}
	In Step 2, we show that this implies
	\begin{align}\label{eq:ProofL2Continuity}
		M_{\lambda(t)}^\delta f_t\in C([0,T];L^2)
	\end{align}
	and that the following formal argument can be made rigorous. We use $ (M_{\lambda(t)}^\delta)^2f_t $ as a test function to get
	\begin{align*}
		\dfrac{1}{2}\norm[L^2]{M_{\lambda(t)}^\delta f_t}^2 &= \dfrac{1}{2}\norm[H^a]{f_0}^2+ \dfrac{1}{2}\int_0^t\trace L_\tau \, \norm[L^2]{M_{\lambda(\tau)}^\delta f_\tau}^2\, d\tau + \dfrac{1}{2}\int_0^t \int_{\R^3} f_\tau \, \partial_\tau\left[  (M_{\lambda(\tau)}^\delta)^2 \right] f_\tau \, dv\tau
		\\ 
		&- \int_0^t\int_{\R^3} f_\tau \, L_\tau v\cdot \nabla (M_{\lambda(\tau)}^\delta)^2f_\tau\, dvd\tau + \int_0^t\int_{\R^3} Q(f_\tau,f_\tau) (M_{\lambda(\tau)}^\delta)^2f_\tau\, dvd\tau.
	\end{align*}
	The last two terms make sense when we use commutators. This leads to
	\begin{align}\label{eq:ProofWeakForm}
		\begin{split}
			\dfrac{1}{2}\norm[L^2]{M_{\lambda(t)}^\delta f_t}^2 &= \dfrac{1}{2}\norm[H^a]{f_0}^2 + N\int_0^t \int_{\R^3} \left( \sqrt{\log \left\langle D_v \right\rangle} 	M_{\lambda(\tau)}^\delta f_\tau \right)^2 \, dvd\tau
			\\ 
			&+ \int_0^t\int_{\R^3} \left( C_\tau f_\tau\right)  \, \left( M_{\lambda(\tau)}^\delta f_\tau\right) \, dvd\tau + \int_0^t \dualbra{Q(f_\tau,M_{\lambda(\tau)}^\delta f_\tau)}{M_{\lambda(\tau)}^\delta f_\tau} \, d\tau 
			\\
			&+\int_0^t \dualbra{M_{\lambda(\tau)}^\delta Q(f_\tau,f_\tau)-Q(f_\tau,M_{\lambda(\tau)}^\delta f_\tau)}{M_{\lambda(\tau)}^\delta f_\tau}\, d\tau.
		\end{split}
	\end{align}
	Here, we introduced the commutator 
	\begin{align*}
		C_t(D_v):=L_t v\cdot  \nabla M_{\lambda(t)}^\delta(D_v) -  \nabla M_{\lambda(t)}^\delta(D_v) \cdot L_t v.
	\end{align*} 
	For formula \eqref{eq:ProofWeakForm} we used several observations. First of all, we applied
	\begin{align*}
		\partial_\tau M_{\lambda(\tau)}^\delta = N\log\left\langle \xi \right\rangle M_{\lambda(\tau)}^\delta.
	\end{align*}
	For the drift term we used
	\begin{align*}
		&\dfrac{1}{2}\int_0^t\trace L_\tau \, \norm[L^2]{M_{\lambda(\tau)}^\delta f_\tau}^2\, d\tau-\int_0^t\int_{\R^3} f_\tau \, L_\tau v\cdot \nabla (M_{\lambda(\tau)}^\delta)^2f_\tau\, dvd\tau
		\\
		&= \dfrac{1}{2}\int_0^t\trace L_\tau \, \norm[L^2]{M_{\lambda(\tau)}^\delta f_\tau}^2\, d\tau+\int_0^t\int_{\R^3} \left( C_\tau f_\tau\right)  \, \left( M_{\lambda(\tau)}^\delta f_\tau\right) \, dvd\tau 
		\\
		&\qquad + \int_0^t\int_{\R^3} \left[ M_{\lambda(\tau)}^\delta f_\tau\right] \, L_\tau v\cdot \nabla\left[ M_{\lambda(\tau)}^\delta f_\tau \right] \, dvd\tau.
	\end{align*}
	In the last expression we apply partial integration so that the term involving $ \trace L_t $ cancels. 
	
	In formula \eqref{eq:ProofWeakForm} all terms make sense due to \eqref{eq:Step1ProofApproxRegularity}. In fact, using the Fourier transform one can show that the symbol of $ C_t $ is given by $ -L^\top_t\xi\cdot \nabla_\xi M^\delta_{\lambda(t)} $, which is bounded by $ C\norm{L_t} M^\delta_{\lambda(t)} $ uniformly in $ \delta>0 $. The terms involving the collision operator make sense due to Lemma \ref{lem:CoercivityUpperBound} (ii) and Lemma \ref{lem:CommutatorErstimate} in conjunction with \eqref{eq:Step1ProofApproxRegularity}. 
	
	In \eqref{eq:ProofWeakForm} we use the previous observations and Lemma \ref{lem:CoercivityUpperBound} (i) to obtain
	\begin{align*}
		\dfrac{1}{2}\norm[L^2]{M_{\lambda(t)}^\delta f_t}^2 &\leq \dfrac{1}{2}\norm[H^a]{f_0}^2 + C\norm[L^\infty]{L}\int_0^t \norm[L^2]{M_{\lambda(\tau)}^\delta f_\tau}^2\, d\tau  +N\int_0^t\norm[L^2]{\sqrt{\log \left\langle D_v \right\rangle}  M_{\lambda(\tau)}^\delta f_\tau}^2 \, d\tau 
		\\
		&+ C_f\int_0^t\norm[H^{s'}_{\gamma/2}]{M_{\lambda(\tau)}^\delta f_\tau}^2 \, d\tau - c_f \int_0^t \norm[H^s_{\gamma/2}]{M_{\lambda(\tau)}^\delta f_\tau}^2 \, d\tau + C_f\int_0^t\norm[L^2]{M_{\lambda(\tau)}^\delta f_\tau}^2 \, d\tau.
	\end{align*}
	For the second, third, fourth and last term we use interpolation in Sobolev spaces to get
	\begin{align*}
		\sup_{t\in[0,T_1]}	\norm[L^2]{M_{\lambda(t)}^\delta f_t}^2  \leq C(T_1).
	\end{align*}
	This implies, with interpolation in weighted Sobolev spaces, see \cite[Lemma 3.9]{AlexandreMorimoto2012SmoothingEffectWeakSolHomogBE}, for $ s_0:=(1-s)/3 $
	\begin{align*}
		f\in L^\infty([T_1/2,T_1], H^{s_0+a}_\ell),
	\end{align*}
	since $ 0<s_0<\lambda(T_1/2)=(1-s)/2+a $. Since $ T_1>0 $ was arbitrary we conclude for any $ t_0>0 $ and $ \ell\geq0 $
	\begin{align*}
		f\in L^\infty([t_0,T], H^{s_0+a}_\ell).
	\end{align*} 
	The regularity improved by the fixed amount $ s_0>0 $ so that we can repeat the above reasoning inductively, starting at some new time $ t_0>0 $. Let us comment here on the use of Lemma \ref{lem:CommutatorErstimate}. Here, we needed to distinguish two cases. In the second case, when $ \lambda(t)+s'\geq 3/2 $ the constant $ C_f $ above includes the norm $ \norm[H^{(\lambda(t)+s'-3)_+}]{f_t} $. Let us verify that this is bounded in the induction. In the $ k $-th step we have $ \lambda(t):=Nt+ks_0 $ for $ t\leq T_1 $, $ NT_1=1-s $. Hence, we get
	\begin{align*}
		\lambda(t)+s'-3 \leq ks_0-2-s+s' \leq ks_0-2.
	\end{align*}
	Thus, this term is bounded due to the $ (k-1) $-th step.
	
	To prove that $ f\in L^\infty([0,T]; L^1_p) $ for all $ p\geq0 $ implies  $ f\in L^\infty([T_0,T];L^2_p) $ for all $ p\geq0 $ and any $ T_0\geq0 $, one can follow the arguments in \cite[Theorem 5.1]{AlexandreMorimoto2012SmoothingEffectWeakSolHomogBE}. Here, one starts the induction with the regularity
	\begin{align*}
		f\in L^\infty([T_0,T];H^{-3/2-\varepsilon}_p)
	\end{align*}
	for any $ p\geq0 $, $ \varepsilon>0 $ due to the embedding $ L^1_p\subset H^{-3/2-\varepsilon}_p $. Furthermore, one chooses $ \lambda(t) $ and $ N_0 $ such that for any $ p\geq0 $
	\begin{align*}
		M_{\lambda(t)}^\delta f_{t'}\in L^\infty ([0,T]\times [0,T]; H^{s_1}_p)
	\end{align*}
	and some $ s_1>s $. This regularity allows to make the corresponding computations rigorous, see Step 2. For more details see \cite[Theorem 5.1]{AlexandreMorimoto2012SmoothingEffectWeakSolHomogBE}.
	
	\textit{Step 2:} Finally, we show that the above formal computations can be made rigorous. This corresponds to \cite[Lemma 4.3]{AlexandreMorimoto2012SmoothingEffectWeakSolHomogBE}. We prove that, with the notation as in Step 1,
	\begin{align}\label{eq:ProofApproxRegularity}
		M_{\lambda(t)}^\delta f_{t'}\in L^\infty ([0,T]\times [0,T]; H^{s_1}_{\ell})
	\end{align}
	for all $ \ell\geq0 $ and some $ s_1>s $ implies \eqref{eq:ProofL2Continuity} and \eqref{eq:ProofWeakForm}.  We first show \eqref{eq:ProofL2Continuity} and due to the drift term a regularization is necessary. Let us define for $ \kappa>0 $
	\begin{align*}
		M_{\lambda(t)}^{\delta,\kappa}(D_v) = \dfrac{1}{1+\kappa \left\langle D_v \right\rangle }M_{\lambda(t)}^{\delta}(D_v).
	\end{align*}
	For $ t'<t $ we choose $ (M_{\lambda(\bar{t})}^{\delta,\kappa})^2f_{\bar{t}} $ with $ \bar{t}=t',\, t $ as time-independent test function. We can do this by an approximation $ (M_{\lambda(\bar{t})}^{\delta,\kappa})^{-1} \psi_j\to M_{\lambda(\bar{t})}^{\delta,\kappa} f_{\bar{t}} $ in $ H^s_{\ell_0} $ for $ \psi_j\in C_0^\infty $. In the corresponding expressions we use then again commutator estimates. For the collision operator, we apply Lemma \ref{lem:CommutatorErstimate} and Lemma \ref{lem:CoercivityUpperBound} (ii). For the drift term, we use the commutator
	\begin{align*}
		L_t v\cdot  \nabla M_{\lambda(\bar{t})}^{\delta,\kappa}(D_v) -  \nabla M_{\lambda(\bar{t})}^{\delta,\kappa}(D_v) \cdot L_t v.
	\end{align*} 
	As in Step 1, the corresponding symbol is bounded by $ CM_{\lambda(\bar{t})}^{\delta,\kappa} $ uniformly in $ \kappa,\, \delta>0 $. The two results for $ \bar{t}=t',\, t $ are added two yield
	\begin{align*}
		\norm[L^2]{M_{\lambda(t)}^{\delta,\kappa}f_{t}}^2-\norm[L^2]{M_{\lambda(t')}^{\delta,\kappa}f_{t'}}^2 = \int_{\R^3}f_t\, \left( (M_{\lambda(t)}^{\delta,\kappa})^2-(M_{\lambda(t')}^{\delta,\kappa})^2 \right) \, f_{t'}\, dv + \mathcal{O}(|t-t'|).
	\end{align*}
	The last part contains all the remaining terms, which are time-integrals with bounded integrand. The first term on the right-hand is seen to be also of order $ \mathcal{O}(|t-t'|) $. All of the terms are uniformly bounded in $ \kappa>0 $, so we let $ \kappa\to 0 $. This shows 
	\begin{align*}
		\lim_{t'\to t}\norm[L^2]{M_{\lambda(t')}^{\delta}f_{t'}}^2=\norm[L^2]{M_{\lambda(t)}^{\delta}f_{t}}^2.
	\end{align*}
	If we take the differences of the expressions for $ \bar{t}=t',\, t $ we get
	\begin{align*}
		\lim_{t'\to t}\int_{\R^3} \left( M_{\lambda(t')}^{\delta}f_{t'} \right) \left( M_{\lambda(t)}^{\delta}f_{t} \right) \, dv = \norm[L^2]{M_{\lambda(t)}^{\delta}f_{t}}^2.
	\end{align*}
	This implies \eqref{eq:ProofL2Continuity}. 
	
	To prove that \eqref{eq:ProofWeakForm} is rigorous, we divide $ [0,t] $ into time steps $ t_j=jt/k $, $ j=0,\ldots, k $. As above we want to use $ (M_{\lambda(\bar{t})}^{\delta,\kappa})^2f_{\bar{t}} $ with $ \bar{t}=t_{j-1},\, t_j $ as time-independent test function. Subtracting the resulting expressions and adding in $ j=0,\ldots, k $, we aim to letting $ k\to \infty $. This would lead to \eqref{eq:ProofWeakForm}. To this end, the integrands in the time-integrals have to be continuous in $ t $. In fact, we have 
	\begin{align*}
		M_{\lambda(t)}^\delta f_t\in C([0,T];H^s_\ell)
	\end{align*}
	as a consequence of \eqref{eq:ProofL2Continuity} and interpolation with the estimate \eqref{eq:ProofApproxRegularity}. This is enough to conclude.
\end{proof}

\section{Collision dominated behavior for a model equation}
\label{sec:CollDomBehav}

In this section, we study a rescaling of solutions $ g $ to \eqref{eq:homenergBE}. As in the introduction we consider $ f_t(v)=g_t(v\beta_t^{-1/2}+V_t)\beta_t^{-3/2}\rho_t^{-1} $ with $ \rho_t,\,  V_t,\, \beta_t $ given in \eqref{eq:MassMomentum}. This yields a solution to \eqref{eq:IndtrodModelHomoengBE}. However, the matrix $ L_t $ might be unbounded in time, as for instance in the case of combined orthogonal shear \eqref{eq:CombOrthShear}. In the analysis here, it will be convenient to reduce it to the case of bounded matrices $ L $. To this end, we use a change of time, see the proof of Theorem \ref{thm:HomoenergticSol1} in Section \ref{sec:Application}. Such a transformation yields a solution $ f $ to equations of the form (cf. \eqref{eq:IndtrodModelHomoengBE})
\begin{align}\label{eq:ModelEq}
	\begin{split}
		\partial_t f &= \div\left( \left( L_t-\alpha_t\right) v\, f \right) + \nu_t  \beta_t^{-\gamma/2}Q(f,f)
		\\
		\beta_t &= \beta_0 \exp\left( 2\int_0^t\alpha_s\, ds \right) \quad \alpha_t := \dfrac{1}{3}\int_{\R^3}v\cdot L_tv\, f_t(v)\, dv.
	\end{split}
\end{align}
In these equations, the positive function $ \nu $ and the matrix-valued function $ L $ are given. In the sequel, we study solutions to \eqref{eq:ModelEq}. Note that our investigations in the last section yields corresponding well-posedness and regularity results for \eqref{eq:ModelEq}.

Let us introduce the decomposition $ L_t = A_t +b_t \, I $ into the trace-free and trace part, $ b_t := \trace L_t/3 $. We study solutions to \eqref{eq:ModelEq} of the form
\begin{align}\label{eq:AnsatzHilbertExpansion}
	f_t(v)=\mu(v)+\bar{\mu}_t(v)+h_t(v).
\end{align}
The term $ \bar{\mu} $ corresponds to the first order term in the Hilbert-type expansion, cf. \eqref{eq:HilbertExpFirstOrder} and \eqref{eq:HilbertExpFirstOrder2}. Due to the term $ \nu $ it has the form
\begin{align}\label{eq:HilbertExpFirstOrder3}
	\bar{\mu}_t := \dfrac{1}{\eta_t}\Lin^{-1}\left[ -v\cdot A_t v \mu \right], \quad \eta_t:= \nu_t\, \beta_t^{-\gamma/2}.
\end{align} 
As in the introduction, let us determine the asymptotics of the inverse temperature $ \beta_t $, which yields the behavior of $ \eta_t= \nu_t \beta_t^{-\gamma /2} $. The inverse temperature $ \beta_t $ satisfies the equation, see \eqref{eq:ModelEq},
\begin{align*}
	\dfrac{\beta_t'}{2\beta_t} = \alpha_t = \dfrac{1}{3}\int_{\R^3} v\cdot L_t v\cdot f_t(v)\, dv.
\end{align*}
If we consider only the first two terms in \eqref{eq:AnsatzHilbertExpansion} we obtain the equation
\begin{align}\label{eq:AsympEqInverseTemp}
	\left( \beta_t^{-\gamma/2} \right)' = -\gamma\,  b_t\, \beta_t^{-\gamma/2} + \dfrac{\gamma\, a_t}{3 \nu_t},
\end{align}
where we defined 
\begin{align}\label{eq:ConstantFirstOrderHilbertExp}
	a_t:= \dualbra{v\cdot A_tv \, \mu}{ \Lin^{-1} \left[ v\cdot A_tv \, \mu\right] }_{L^2(\mu^{-1/2})}>0.
\end{align}
The solution to \eqref{eq:AsympEqInverseTemp} is given by
\begin{align*}
	B_t(\beta_0):=\beta_0^{-\gamma/2}\, e^{-\gamma\int_0^tb_s \, ds} + \int_0^t\dfrac{\gamma\, a_s  }{3\nu_s}\,e^{-\gamma\int_s^tb_r \, dr} \, ds.
\end{align*}
If $ h $ amounts to an error which is integrable in time, the behavior of $ \eta_t= \nu_t \beta_t^{-\gamma/2} $ is determined by the function
\begin{align}\label{eq:GrowthInversTemp}
	Z_t(\beta_0):= \nu_t B_t(\beta_0) = \beta_0^{-\gamma/2}\,\nu_t e^{-\gamma\int_0^tb_s \, ds} + \dfrac{\gamma}{3}\int_0^t\dfrac{\nu_t\, a_s  }{\nu_s}\,e^{-\gamma\int_s^tb_r \, dr} \, ds.
\end{align}
One crucial assumption in the theorem below is a growth condition on $ Z_t(\beta_0) $. As we will see in Section \ref{sec:Application}, this condition is always satisfied for homoenergetic solutions.

\begin{thm}\label{thm:MainTheorem}
	Consider equations \eqref{eq:ModelEq} under the following structural assumptions.
	\begin{enumerate}[(I)]
		\item The matrix $ L_t\in C^1([0,\infty); \R^{3\times 3}) $ satisfies $ \sup_{t\geq 0}\left\lbrace \norm{L_t}+\norm{L_t'} \right\rbrace <\infty $. Furthermore, for $ \nu\in C^1([0,\infty);(0,\infty)) $ it holds $ \sup_{t\geq 0}|\nu_t'/\nu_t|<\infty $.
		\item We assume that $ Z_t(1) \approx 1+t $ for all $ t\geq 0 $, where $ Z_t $ is given in \eqref{eq:GrowthInversTemp}.
	\end{enumerate}
	Let $ p_0> 4+4s +3/2 $ be arbitrary and $ f_0\in \mathcal{H}^1_{p_0} $ satisfying the normalization \eqref{eq:Normalization}. Consider the unique solution $ f $ to \eqref{eq:ModelEq} and define $ h_t := f_t - \mu- \bar{\mu}_t $ with $ \bar{\mu} $ given in \eqref{eq:HilbertExpFirstOrder3}.
	
	There are $ \varepsilon_0\in(0,1) $ sufficiently small and a constant $ C' $, depending only on $ p_0 $, $ L_t,\, \nu_t $ and the collision kernel $ B $, such that the following holds. 
	
	If $ \norm[\mathcal{H}^1_{p_0}]{h_0}=\varepsilon \leq \varepsilon_0 $ and $ \beta_0\leq \varepsilon_0 $, then we have for all $ t\geq 0 $
	\begin{align}\label{eq:SecondOrderHilbertExpDecay}
		\norm[\mathcal{H}^1_{p_0}]{h_t}\leq C'\left( \dfrac{\varepsilon}{(1+t)^2}+\dfrac{1}{Z_t(\beta_0)^2}\right) .
	\end{align}
	In addition, for any $ k,\, p\in\N $ and $ t\geq0 $ it holds
	\begin{align}\label{eq:FirstOrderHilbertExpDecay}
		\norm[H^k_p]{\dfrac{\bar{\mu}_t}{\sqrt{\mu}}} \leq \dfrac{C_{k,p}}{Z_t(\beta_0)}.
	\end{align}
	Finally, we have
	\begin{align}\label{eq:InvTempAsyBehav}
		\dfrac{1}{4}Z_t(\beta_0)\leq \eta_t=\nu_t \,\beta_t^{-\gamma/2} \leq 4Z_t(\beta_0).
	\end{align}
\end{thm}
\begin{rem}
	In the above theorem, we merely obtain \eqref{eq:InvTempAsyBehav} concerning the inverse temperature. The precise asymptotics of $ \beta_t $ can be calculated with equation \eqref{eq:AsympEqInverseTemp}, when more information (than just assumption $ (II) $) on the functions $ b_t, \, a_t $ and $ \nu_t $ is available. We do this in the proof of Theorem \ref{thm:HomoenergticSol1} in Section \ref{sec:Application}. As we will see, the equation satisfied by $ h $ contains a source term, see \eqref{eq:ErrorEq}. This source term leads to a decay of order $ 1/Z_t(\beta_0)^2 $. The other terms lead to a decay of order $ \varepsilon/(1+t)^2 $. Hence, the combination of both yields \eqref{eq:SecondOrderHilbertExpDecay}.
\end{rem}
\begin{rem}
	In assumption $ (II) $ one can also assume $ Z_t(1)\approx (1+t)^r $ with $ r>1/2 $. The reason for the latter condition is that $ h_t $ will then be of order $ \mathcal{O}(t^{-2r}) $, which is integrable in time.
\end{rem}

\subsection{Collision-dominated analysis}
\label{subsec:CollDomAnalysis}
In this section we prove Theorem \ref{thm:MainTheorem} and for this reason prove several estimates.

\paragraph{Preparation.} Due to $ f_0\in \mathcal{H}^1_{p_0} $ there is a unique solution $ f $ to \eqref{eq:ModelEq} by Proposition~\ref{pro:Wellposedness}, which is smooth for positive times by Proposition \ref{pro:Regularity}. Moreover, note that the norm $ t\mapsto \norm[H^1_{p_0}]{f_t} $ is continuous for $ t\geq0 $. As in Theorem \ref{thm:MainTheorem} we set $ h_t=f_t-\mu-\bar{\mu}_t $. Correspondingly, let us decompose
\begin{align}\label{eq:DecompositionLagrangeMult}
	\begin{split}
		\alpha_t &= \alpha_t^0+\alpha_t^1+\alpha_t^2,
		\\
		\alpha_t^0&=\dfrac{1}{3}\int_{\R^3}v\cdot L_tv \, \mu \, dv = b_t, \quad \alpha_t^1=\dfrac{1}{3}\int_{\R^3}v\cdot L_tv \, \bar{\mu}_t \, dv, \quad \alpha_t^2=\dfrac{1}{3}\int_{\R^3}v\cdot L_tv \, h_t \, dv.
	\end{split}
\end{align}
In order to obtain the equation solved by $ h $, we plug the expansion $ f_t=\mu+\bar{\mu}_t+h_t $ into \eqref{eq:ModelEq}. Recall that $ \eta_t=\nu_t\beta_t^{-\gamma/2} $. We obtain with
\begin{align*}
	0 = \div\left( (L_t-\alpha_t^0)v \, \mu \right) - \eta_t \Lin \bar{\mu}_t = \div\left( A_tv \, \mu \right) - \eta_t \Lin \bar{\mu}_t,
\end{align*}
by definition of $ \bar{\mu} $, cf. \eqref{eq:HilbertExpFirstOrder3}, the equation
\begin{align}\label{eq:ErrorEq}
	\begin{split}
		\partial_th =& \big[ -\partial_t \bar{\mu} +\div(A_tv \, \bar{\mu}) - \alpha_t^1\, \div \left( v\, (\mu+\mu^1)\right) +\eta_t\, Q(\bar{\mu},\bar{\mu}) \big]
		\\
		&+ \big[ -\alpha_t^2 \, \div \left( v \, (\mu+\bar{\mu}) \right) + \div(A_tv \ h)-\alpha^1_t \, \div(v \, h) + \eta_t\, \left( Q(\bar{\mu},h)+Q(h,\bar{\mu}) \right) \big]
		\\
		&- \alpha^2_t\, \div(v \, h) + \eta_t \, Q(h,h)-\eta_t \,  \Lin h
		\\
		&=: S+\mathscr{R}h- \alpha^2_t\, \div(v \, h)+\eta_t \, Q(h,h)-\eta_t \, \Lin h.
	\end{split}
\end{align} 
Observe that $ h_t\mapsto \alpha^2_t $ is linear in the definition of $ \mathscr{R} $ and that $ S $, $ \mathscr{R} $ are time-dependent. Note also that $ h_t\in(\ker\Lin)^\perp $, since $ \bar{\mu}_t\in(\ker\Lin)^\perp $ and the fact that $ f_t $ has the same mass, momentum and energy as the Maxwellian $ \mu $. In particular, this implies that $ Q(h_t,h_t)\in(\ker\Lin)^\perp $ and
\begin{align*}
	S+\mathscr{R}h_t- \alpha^2_t\, \div(v \, h_t) \in(\ker\Lin)^\perp
\end{align*}
for all $ t\geq0 $.

\paragraph{Strategy.} The proof relies on the following two estimates:
\begin{align}
		\dfrac{1}{4}Z_t(\beta_0) &\leq \eta_t=\nu_t\beta_t^{-\gamma/2} \leq 4Z_t(\beta_0), \label{eq:LoopInequality1}
		\\
		\norm[\mathcal{H}^1_{p_0}]{h_t} &\leq \Omega\left( \dfrac{\varepsilon}{(1+t)^2}+\dfrac{1}{\eta_t^2}\right)  \label{eq:LoopInequality2}
\end{align}
for some constant $ \Omega>0 $. Note that both imply \eqref{eq:SecondOrderHilbertExpDecay}.

Note that $ \eqref{eq:LoopInequality1} $ together with assumption $ (II) $ in Theorem \ref{thm:MainTheorem} implies for $ \beta_0\leq 1 $
\begin{align}\label{eq:GrowthCollisionConstant}
	\eta_t\geq \dfrac{1}{4}Z_t(\beta_0) \geq \dfrac{1}{8}\bar{\eta}(\beta_0) +\dfrac{1}{8}Z_t(1) \geq c_0(\bar{\eta}(\beta_0)+t).
\end{align}
Here, the constant $ c_0>0 $ does not depend on $ \beta_0 $ and we defined
\begin{align*}
	\bar{\eta}(\beta_0):= \min_{t\geq0}Z_t(\beta_0).
\end{align*}
In fact, we assume $ \beta_0\leq \varepsilon_0\in(0,1) $ sufficiently small as in Theorem \ref{thm:MainTheorem}, so that $ \beta_0\leq 1 $ is always satisfied. Furthermore, note that $ \bar{\eta}(\beta_0)\to \infty $ as $ \beta_0\to 0 $. Hence, choosing $ \beta_0 $ small ensures that the factor $ \eta_t $ in front of the collision operator is large for all times. As a consequence the collision operator is always the dominant term.

We will prove that \eqref{eq:LoopInequality1} implies \eqref{eq:LoopInequality2} and the other way round. However, in order to choose the constant $ \Omega $ such that this loop can be closed, we need to take $ \varepsilon_0 $ sufficiently small. Recall that $ \norm[\mathcal{H}^1_{p_0}]{h_0}=\varepsilon\leq \varepsilon_0 $ and $ \beta_0\leq\varepsilon_0 $. 

\subsubsection{Estimates on $\bar{\mu}$ and $\eta_t$.}
Let us first give the following regularity properties for the first order approximation $ \bar{\mu} $ defined in \eqref{eq:HilbertExpFirstOrder}, which implies \eqref{eq:FirstOrderHilbertExpDecay}.
\begin{lem}\label{lem:FirstOrderEstimate}
	The function $ \bar{\mu}_t $ defined above satisfies for all $ q,k\in\N $
	\begin{align}\label{eq:LemFirstOrderEstimate1}
		\norm[H^k_q]{\bar{\mu}_t/\sqrt{\mu}} &\leq C_{q,k}\dfrac{\norm{A}}{\eta_t} \leq \dfrac{C_{q,k}}{\eta_t},
		\\\label{eq:LemFirstOrderEstimate2}
		\norm[H^k_q]{\partial_t \bar{\mu}_t/\sqrt{\mu}} &\leq C_{q,k}\dfrac{1}{\eta_t}\left( \dfrac{\norm{A}|\eta_t'|}{\eta_t}+\norm{\partial_tA} \right)\leq \dfrac{C_{q,k}}{\eta_t}.
	\end{align}
\end{lem}
\begin{proof}
	First of all, note that $ \tilde{\mu}_t:=\bar{\mu}_t/\sqrt{\mu} $ satisfies the equation
	\begin{align}\label{eq:ProofLemFirstOrderEstimate}
		L\tilde{\mu}_t = \dfrac{1}{\eta_t}\, v\cdot A_tv\, \sqrt{\mu}
	\end{align}
	due to \eqref{eq:HilbertExpFirstOrder3}. Here, we used the notation $ Lg=\mu^{-1/2}\Lin[\sqrt{\mu}g] $. As mentioned in the introduction, the operator $ L $ is non-negative, self-adjoint on $ L^2(\R^3) $ with spectral gap (since here $ \gamma>0 $). It coincides with the operator $ \Lin $ on $ L^2(\mu^{-1/2}) $. Corresponding coercivity estimates are available in \cite{AlexandreEtAl2011GlobalExistenceFullRegBoltzmannEq,GressmanStrain2011GlobalClassicalSolBoltzmannEq}, in particular see \cite[Lemma 2.6]{GressmanStrain2011GlobalClassicalSolBoltzmannEq}. (In fact, in the case of $ \gamma>0 $ the operator $ L $ has a regularizing effect both in terms of weights and Sobolev regularity.) Since $ v\cdot A_tv\, \sqrt{\mu}\in H^k_p $ for all $ k, \, p\in\N $, these coercivity estimates allow to prove that $ \tilde{\mu}_t\in H^k_p $ for all $ k, \, p\in\N $ with the asserted bound in \eqref{eq:LemFirstOrderEstimate1}.
	
	The estimate \eqref{eq:LemFirstOrderEstimate2} follows from differentiating equation \eqref{eq:ProofLemFirstOrderEstimate} with respect to time. Note that $ \partial_t[v\cdot A_tv\, \sqrt{\mu}/\eta_t]\in (\ker L)^\perp $, since for all $ t\geq0 $ we have $ v\cdot A_tv\, \sqrt{\mu}/\eta_t\in (\ker L)^\perp $. Hence, the corresponding equation has a unique solution. The coercivity estimates mentioned before allow to prove the asserted bounds. Finally, note that, due to assumption $ (I) $ in Theorem~\ref{thm:MainTheorem}, we have $ \norm[C^1(0,\infty)]{A}<\infty $ and
	\begin{align}\label{eq:ProofTimeIntegralBound}
		\left| \dfrac{\eta_t'}{\eta_t} \right| \leq \left| \dfrac{\nu_t'}{\nu_t}  \right| + \gamma \left| \dfrac{\beta_t'}{2\beta_t} \right| \leq C(1+|\alpha_t|) \leq C(1+\norm[L^1_2]{f_t})\leq C.
	\end{align}
\end{proof}
Let us now prove that \eqref{eq:LoopInequality2} implies \eqref{eq:LoopInequality1}.
\begin{lem}\label{lem:LoopInequality1}
	Assume that $ h_t $ satisfies \eqref{eq:LoopInequality2} for a constant $ \Omega>0 $ on some interval $ [0,T] $. Then, we have on $ [0,T] $
	\begin{align*}
		\exp\left( -c\gamma \, 
		\Omega\, R_T(\varepsilon,\beta_0)\right)\, Z_t(\beta_0)
		&\leq \eta_{t} \leq \exp\left( c\gamma \, \Omega\, R_T(\varepsilon,\beta_0) \right) \, Z_t(\beta_0),
		\\
		R_T(\varepsilon,\beta_0)&:=\int_0^T\left( \dfrac{\varepsilon}{(1+t)^2}+\dfrac{1}{\eta_t^2} \right) \, dt
	\end{align*}
	for some constant $ c>0 $.
\end{lem}
In the final step, we will choose $ \varepsilon,\, \beta_0\leq \varepsilon_0 $ and hence $ R_T(\varepsilon,\beta_0) $ small enough to close the continuation argument.
\begin{proof}[Proof of Lemma \ref{lem:LoopInequality1}]
	Due to the equation
	\begin{align*}
			\dfrac{\beta_t'}{2\beta_t} = \alpha_t = \alpha_t^0+\alpha_t^1+\alpha_t^2
	\end{align*}
	we conclude with \eqref{eq:DecompositionLagrangeMult}
	\begin{align}\label{eq:ProofInversTempEq}
		\dfrac{d}{dt}\beta_t^{-\gamma/2} = -\gamma \beta_t^{-\gamma/2} + \dfrac{\gamma\, a_t}{3\nu_t} -\gamma \,  \alpha_t^2 \, \beta_t^{-\gamma/2}.
	\end{align}
	Recall the definition of $ a_t $ in \eqref{eq:ConstantFirstOrderHilbertExp}. Observe that due to \eqref{eq:LoopInequality2}
	\begin{align*}
		\alpha_t^2= \dfrac{1}{3}\int_{\R^3}v\cdot L_tv\, h_t\, dv \leq \norm[L^\infty]{L}\norm[L^1_2]{h_t}\leq c\norm[\mathcal{H}^1_{p_0}]{h_t}
	\end{align*}
	and hence
	\begin{align*}
		\int_0^T \alpha^2_t\, dt \leq c \, 
		\Omega\, R_T(\varepsilon,\beta_0).
	\end{align*}
	We integrate \eqref{eq:ProofInversTempEq} to obtain the claim.
\end{proof}

\subsubsection{Estimate on the error term}
Here, we prove that \eqref{eq:LoopInequality1} implies \eqref{eq:LoopInequality2}. This is more involved and relies on several estimates and known results. We split the analysis in the estimates in the $ L^2 $-framework and the estimates in the $ L^1 $-framework. To this end, let us define 
\begin{align}\label{eq:Defintionm}
	m:=p_0-2-4s-3/2>2,
\end{align}
which will be used as a weight in the $ L^1 $ estimates. Let us note that $ \norm[L^1_m]{h_0}\leq C_*\varepsilon $ for some constant $ C_* $ by our assumption on $ h_0 $ in Theorem \ref{thm:MainTheorem}.

\paragraph{Estimates in $ L^2 $-framework.} Here we discuss the estimates of solutions $ h $ to \eqref{eq:ErrorEq} in the space $ \mathcal{H}^1_{p_0} $. Due to the angular singularity in the collision operator, we are led to use the following (homogeneous) anisotropic norm 
\begin{align*}
		\norm[\dot{H}^{s,*}]{g}^2 := \int_{\R^3}\int_{\R^3}\int_{S^2} b_\delta(n\cdot \sigma )\mu_* \left\langle v_* \right\rangle^{-\gamma} \left( g' \left\langle v' \right\rangle^{\gamma/2} -  g \left\langle v \right\rangle^{\gamma/2} \right) ^2\, d\sigma dv_*dv.
\end{align*}
Here, we defined $ b_\delta(\cos\theta) = \chi(\theta/\delta)b(\cos\theta) $ for a smooth function $ \chi $ with $ \ind_{[-1,1]}\leq \chi \leq \ind_{[-2,2]} $. The parameter $ \delta>0 $ will be fixed such that Lemma \ref{lem:RegEstLinearizedCollOp} holds. Finally, we also define the following weighted anisotropic norm
\begin{align*}
		\norm[H^{s,*}_p]{g}^2= \norm[L^2_{p+\gamma/2}]{g}^2 + \norm[\dot{H}^{s,*}]{g\left\langle \cdot \right\rangle^p}^2.
\end{align*}
\begin{rem}
	Let us mention that the above anisotropic norm has been introduced in  \cite{HerauTononTristani2020RegEstCauchyInhomBoltzmann} following works by Alexandre et. al. (see \cite{AlexandreEtAl2011GlobalExistenceFullRegBoltzmannEq}). A different anisotropic norm was used in \cite{GressmanStrain2011GlobalClassicalSolBoltzmannEq}. We refer to \cite{He2018SharpBoundsBoltzmannLandau} for a discussion and further estimates of the Boltzmann collision operator in anisotropic spaces.
\end{rem}
The following estimate relates the space $ H^{s,*}_p $ to the standard fractional Sobolev spaces, cf. \cite[Lemma 2.1]{HerauTononTristani2020RegEstCauchyInhomBoltzmann}.
\begin{lem} 
	For any $ k\geq 0 $ and all $ g\in H^s_{k+\gamma/2+s} $ it holds
	\begin{align*}
			\delta^{2-2s}\norm[H^s_{k+\gamma/2}]{g}\lesssim \norm[H^{s,*}_{k}]{g} \lesssim \norm[H^s_{k+\gamma/2+s}]{g}.
	\end{align*}
\end{lem}
We fix $ \delta>0 $ in the definition of the norm of $ \dot{H}^{s,*} $ such that the following lemma holds, cf. \cite[Lemma 4.2]{HerauTononTristani2020RegEstCauchyInhomBoltzmann}.
\begin{lem}\label{lem:RegEstLinearizedCollOp}
	Let $ k\geq \gamma/2+3+2s $, then for sufficiently small $ \delta>0 $ it holds
	\begin{align*}
			-\dualbra{\Lin h}{h}\leq -c_\delta\norm[H^{s,*}_k]{h}^2 + C_\delta \norm[L^2]{h}^2
	\end{align*}
	for some constants $ c_\delta, \, C_\delta>0 $ depending on $ \delta>0 $.
\end{lem}
Let us also recall the following estimates on the collision operator, cf. \cite[Lemma 2.3]{HerauTononTristani2020RegEstCauchyInhomBoltzmann}.
\begin{lem}\label{lem:L2EstimateCollOp}
	Let $ k>\gamma/2 +2+2s $.
	\begin{enumerate}[(i)]
		\item If $ \ell >\gamma +1+3/2 $ we have
			\begin{align*}
				|\dualbra{Q(f,g)}{h}_{L^2_k}| \lesssim \norm[L^2_\ell]{f}\norm[H^{s_1}_{N_1+k}]{g}\norm[H^{s_2}_{N_2+k}]{h} + \norm[L^2_{\gamma/2+k}]{f}\norm[L^2_{\ell}]{g}\norm[L^2_{\gamma/2+k}]{h},
			\end{align*}
			where $ s_1,\, s_2\in [0,2s] $, $ s_1+s_2=2s $ and $ N_1\geq\gamma/2 $, $ N_2\geq0 $, $ N_1+N_2=\gamma+2s $.
	\item If $ \ell > 4+3/2 $ we have
		\begin{align*}
				|\dualbra{Q(f,g)}{g}_{L^2_k}| \lesssim \norm[L^2_\ell]{f}\norm[H^{s,*}_{k}]{g}^2+ \norm[L^2_{\gamma/2+k}]{f}\norm[L^2_{\ell}]{g}\norm[L^2_{\gamma/2+k}]{g}.
		\end{align*}
	\end{enumerate}
\end{lem}
Finally, let us recall the following interpolation estimate. It can be proved using Fourier transform and a splitting in small respectively large frequencies.
\begin{lem}\label{lem:Interpolation}
	For any $ s> r\geq0 $ we have
	\begin{align*}
		\norm[\dot{H}^r]{g}\lesssim \norm[L^1]{g}^\theta\norm[\dot{H}^s]{g}^{1-\theta}
	\end{align*}
	with $ \theta = (2s-2r)/(2s+3) $.
\end{lem}
Let us now give the first conditional estimate on the error term $ h $.
\begin{pro}\label{pro:PrepLoopInequality2}
	Under the assumptions of Theorem \ref{thm:MainTheorem} there is a constant $ C' $ such that the following holds. Assuming that \eqref{eq:LoopInequality1} and
	\begin{align}\label{eq:PropAssL1}
		\norm[L^1_m]{h}\leq \Omega'\left( \dfrac{\varepsilon}{(1+t)^2}+\dfrac{1}{\eta_t^2}\right).
	\end{align}
	hold on some interval $ t\in[0,T] $ for some constant $ \Omega' $, we can find a sufficiently small constant $ \varepsilon_0'\in (0,1) $ so that: if also $ \varepsilon\leq \varepsilon_0' $ and $ \beta_0\leq \varepsilon_0' $ we have the estimate
	\begin{align}\label{eq:PropH1Est}
			\norm[\mathcal{H}^1_{p_0}]{h_t} \leq C'(\Omega'+1)\left( \dfrac{\varepsilon}{(1+t)^2}+\dfrac{1}{\eta_t^2}\right)
	\end{align}
	for all $ t\in [0,T] $. Here, $ \varepsilon_0' $ depends on $ \Omega' $.
\end{pro}
\begin{proof}
	We split the proof into several steps. We first derive the necessary a priori estimates. Let us write for brevity $ p $ instead of $ p_0 $. Furthermore, it is convenient to use the following norm, which is equivalent to $ \norm[\mathcal{H}^1_{p}]{\cdot} $,
	\begin{align*}
		\tnorm[\mathcal{H}^1_{p}]{h}^2= \norm[L^2_{p}]{h}^2 + \kappa \sum_{|\alpha|= 1}\norm[L^2_{p-2s}]{\partial^\alpha h}^2.
	\end{align*}
	Here, $ \kappa\in(0,1) $ will be chosen small enough, but fixed.
	
	\textit{Step 1:} Using \eqref{eq:ErrorEq} we have
	\begin{align*}
		\dfrac{1}{2}\dfrac{d}{dt}\norm[L^2_{p}]{h}^2 = \dualbra{S}{h}_{L^2_p}+\dualbra{\mathscr{R}h}{h}_{L^2_p} - \alpha_t^2\dualbra{\div(vh)}{h}_{L^2_p} + \eta_t\dualbra{Q(h,h)}{h}_{L^2_p}- \eta_t \dualbra{\Lin h}{h}_{L^2_p}.
	\end{align*}
	We estimate term by term.
	\begin{enumerate}[(i)]
		\item With Lemma \ref{lem:FirstOrderEstimate} and Lemma \ref{lem:L2EstimateCollOp} we obtain
		\begin{align*}
			\dualbra{S_t}{h_t}_{L^2_p} &= \dualbra{-\partial_t \bar{\mu} +\div(A_tv \, \bar{\mu}) - \alpha_t^1\, \div \left( v\, (\mu+\bar{\mu}_t)\right) +\eta_t\, Q(\bar{\mu}_t,\bar{\mu}_t)}{h_t}_{L^2_p} 
			\\
			&\leq C_1\left( \dfrac{1}{\eta_t}+\dfrac{1}{\eta_t^2}\right)  \norm[L^2]{h} \leq \dfrac{C_1}{\eta_t}\norm[L^2]{h_t}.
		\end{align*}
		Recall that $ \norm{A_t} $ is uniformly bounded in time and $ |\alpha_1|\lesssim 1/\eta_t $. Furthermore, we used $ \eta_t\geq c_0\bar{\eta}(\beta_0)\geq 1 $ for $ \beta_0\leq \varepsilon_0' $ small, see \eqref{eq:GrowthCollisionConstant}. This will be used repeatedly in the sequel.
		
		\item By definition of $ \mathscr{R}h $ in \eqref{eq:ErrorEq} we have
		\begin{align*}
			\dualbra{\mathscr{R}h}{h}_{L^2_p} = \dualbra{-\alpha_t^2 \, \div \left( v \, (\mu+\bar{\mu}) \right) + \div(Av \ h)-\alpha^1_t \, \div(v \, h) + \eta_t\, \left( Q(\bar{\mu},h)+Q(h,\bar{\mu}) \right)}{h}_{L^2_p}.
		\end{align*}
		The first term can be estimated using Lemma \ref{lem:FirstOrderEstimate} and the second respectively third via partial integration. This yields and upper bound of the form $ C_2\norm[L^2_p]{h}^2 $. Note that we used $ |\alpha_t^2|\lesssim \norm[L^2_p]{h} $. Using Lemma~\ref{lem:FirstOrderEstimate} and Lemma \ref{lem:L2EstimateCollOp} (i) with $ s_1=2s $, $ s_2=0 $, $ N_1=\gamma+2s $, $ N_2=0 $ we get
		\begin{align*}
			\eta_t\dualbra{Q(h,\bar{\mu})}{h}_{L^2_p}\lesssim \norm[L^2_{p+\gamma/2}]{h}^2.
		\end{align*}
		Note that by $ s\in(0,1/2) $ and our choice of $ p=p_0 > 4 +4s +3/2 $ the Lemma applies. Furthermore, we can choose $ \ell $ with $ p \geq \ell > \gamma +1+3/2 $. For the last term we use Lemma \ref{lem:FirstOrderEstimate} and Lemma \ref{lem:L2EstimateCollOp} (ii) leading to
		\begin{align*}
			\eta_t\dualbra{Q(\bar{\mu},h)}{h}_{L^2_p}\lesssim \norm[L^2_p]{h}\norm[H^{s,*}_p]{h}.
		\end{align*}
		Estimating $ \alpha_t^2\dualbra{\div(vh)}{h}_{L^2_p} $ via partial integration gives
		\begin{align*}
			\dualbra{\mathscr{R}h}{h}_{L^2_p}-\alpha_t^2\dualbra{\div(vh)}{h}_{L^2_p} \leq C_2\norm[H^{s,*}_p]{h}^2.
		\end{align*}
		We also used that $ |\alpha_t^2|\leq |\alpha_t-\alpha_t^0-\alpha_t^1|\leq C_2 $. 
		
		\item For the collision operator we apply Lemma \ref{lem:L2EstimateCollOp} (ii) to get
		\begin{align*}
			\dualbra{Q(h,h)}{h}_{L^2_p} \leq C_3\norm[L^2_p]{h}\norm[H^{s,*}_p]{h}^2,
		\end{align*}
		where we chose $ \ell $ such that $ p\geq \ell >4+3/2 $.
		
		\item Finally, using Lemma \ref{lem:RegEstLinearizedCollOp} we obtain
		\begin{align*}
			-\dualbra{\Lin h}{h}_{L^2_p} \leq c_\delta \norm[H^{s,*}_p]{h}^2 + C_\delta \norm[L^2]{h}^2.
		\end{align*}
		We now apply Lemma \ref{lem:Interpolation} and Young's inequality for the last term to get
		\begin{align*}
			-\dualbra{\Lin h}{h}_{L^2_p} \leq \dfrac{c_\delta}{2}\norm[H^{s,*}_p]{h}^2 + C_\delta'\norm[L^1]{h}^2.
		\end{align*}
	\end{enumerate}
	We summarize the preceding estimates yielding
	\begin{align}\label{eq:ProofL2Est}
		\dfrac{1}{2}\dfrac{d}{dt}\norm[L^2_{p}]{h}^2 \leq -\left( \dfrac{\eta_t\, c_\delta}{2}-C_2-\eta_t \, C_3\norm[L^2_p]{h}\right) \norm[H^{s,*}_p]{h}^2 + \dfrac{C_1}{\eta_t}\norm[L^2]{h}+\eta_tC_\delta'\norm[L^1]{h}^2.
	\end{align}
	Let us now turn to the estimates for $ g^i:= \partial_{i}h $. We abbreviate $ q:=p-2s $. Differentiating equation \eqref{eq:ErrorEq} yields (we denote by $ A^i_t $ the $ i $-th column of the matrix $ A_t $)
	\begin{align*}
		\partial_t g^i = & \partial_iS - \alpha^2_t\partial_{i}\left[ \div(v(\mu+\bar{\mu})) \right] + A^i_t \cdot \nabla h-\alpha_t^1\, g^i + \div\left( (A_t-\alpha^1_t)v\, g^i \right) 
		\\
		& + \eta_t\left( Q(\partial_{i}\bar{\mu},h)+Q(h,\partial_{i}\bar{\mu}) \right) +\eta_t  \left( Q(\bar{\mu},g^i)+Q(g^i,\bar{\mu}) \right)
		\\
		& -\alpha^2_tg^i -\alpha^2_t\div(v\, g^i) + \eta_t\left( Q(h,g^i)+Q(g^i,h) \right) + \eta_t\left( Q(\partial_{i}\mu,h)+Q(h,\partial_{i}\mu) \right) - \eta_t\Lin g^i.
	\end{align*}
	Here we used the well-known identity $ \partial_{i}Q(u,v)=Q(\partial_{i}u,v)+Q(u,\partial_{i}v) $ for functions $ u,\, v $. We now estimate term by term in
	\begin{align*}
		\dfrac{1}{2}\dfrac{d}{dt}\norm[L^2_q]{g^i}^2 = \dualbra{\partial_tg^i}{g^i}_{L^2_q}.
	\end{align*}
	\begin{enumerate}[(i)]
		\item Using Lemma \ref{lem:FirstOrderEstimate} and either $ |\alpha_t^2|\leq \norm[L^2_p]{h} $ or $ |\alpha_t^2|\leq C_4 $ we obtain 
			\begin{align*}
				\dualbra{\partial_iS - \alpha^2_t\partial_{i}\left[ \div(v(\mu+\bar{\mu})) \right] + A^i_t \cdot \nabla h-\alpha_t^1\, g^i + \div\left( (A_t-\alpha^1_t)v\, g^i \right)}{g^i}_{L^2_q} 
				\\
				\leq \dfrac{C_4}{\eta_t}\norm[L^2]{g^i} + C_4\left( \norm[L^2_p]{h}\norm[L^2_q]{g^i}+\norm[L^2_q]{\nabla h}\norm[L^2_q]{g^i}+\norm[L^2_q]{g^i}^2 \right).
			\end{align*}
		\item We apply Lemma \ref{lem:L2EstimateCollOp} (i) with $ s_1=s_2=s $, $ N_1=\gamma/2+2s $, $ N_2=\gamma/2 $ to obtain
		\begin{align*}
			\eta_t &\dualbra{Q(\partial_{i}\bar{\mu}+\partial_{i}\mu,h)+Q(h,\partial_{i}\bar{\mu}+\partial_{i}\mu) }{g^i}_{L^2_q} 
			\\
			&\leq \eta_tC_5\left( \norm[H^s_{q+\gamma/2+2s}]{h}\norm[H^s_{q+\gamma/2}]{g^i} + \norm[L^2_{q+\gamma/2}]{h}\norm[L^2_{q+\gamma/2}]{g^i}\right)
			\\
			&\leq \eta_tC_5\left( \norm[H^s_{p+\gamma/2}]{h}\norm[H^s_{q+\gamma/2}]{g^i} + \norm[L^2_{p+\gamma/2}]{h}\norm[L^2_{q+\gamma/2}]{g^i}\right).
		\end{align*}
		In the last inequality we used $ q+2s=p $.
		
		\item With Lemma \ref{lem:L2EstimateCollOp} we can estimate
		\begin{align*}
			\eta_t  \dualbra{Q(\bar{\mu},g^i)+Q(g^i,\bar{\mu})}{g^i}_{L^2_q}\leq C_6 \norm[H^{s,*}_{q}]{g^i}^2.
		\end{align*}
		
		\item We also have with $ |\alpha^2_t|\leq C_7 $
		\begin{align*}
			\dualbra{-\alpha^2_tg^i -\alpha^2_t\div(v\, g^i)}{g^i}_{L^2_q}\leq C_7\norm[L^2_q]{g^i}^2.
		\end{align*}
		
		\item For the mixed terms $ Q(h,g^i)+Q(g^i,h) $ we use Lemma \ref{lem:L2EstimateCollOp} (ii) to get
		\begin{align*}
			\dualbra{Q(h,g^i)}{g^i}_{L^2_q} \lesssim \norm[L^2_p]{h}\norm[H^{s,*}_q]{g^i}^2+\norm[L^2_q]{g^i}\norm[L^2_{q+\gamma/2}]{h}\norm[L^2_{q+\gamma/2}]{g^i}.
		\end{align*}
		Here, we chose $ \ell $ such that $ q\geq \ell > 4+3/2 $. Applying Lemma \ref{lem:L2EstimateCollOp} (i) with $ s_1=s_2=s $, $ N_1=\gamma/2+2s $, $ N_2=\gamma/2 $ yields
		\begin{align*}
			\dualbra{Q(g^i,h)}{g^i}_{L^2_q} \lesssim \norm[L^2_q]{g^i}\norm[H^s_{q+\gamma+2s}]{h}\norm[H^s_{q+\gamma/2}]{g^i} + \norm[L^2_p]{h}\norm[L^2_{q+\gamma/2}]{g^i}^2.
		\end{align*}
		In total we get
		\begin{align*}
			\dualbra{Q(h,g^i)+Q(g^i,h)}{g^i}_{L^2_q} \leq C_8\left( \norm[L^2_p]{h}\norm[H^{s,*}_q]{g^i}^2+\norm[L^2_q]{g^i}\norm[H^{s,*}_p]{h}\norm[H^{s,*}_q]{g^i} \right).
		\end{align*}
		
		\item Applying Lemma \ref{lem:RegEstLinearizedCollOp} gives
		\begin{align*}
			-\dualbra{\Lin g^i}{g^i} \leq -c_\delta \norm[H^{s,*}_q]{g^i}^2+C_\delta \norm[L^2]{g^i}^2.
		\end{align*}
		For the last term we use Lemma \ref{lem:Interpolation} to get
		\begin{align*}
			\norm[L^2]{g^i} \leq \norm[\dot{H}^1]{h} \lesssim \norm[L^1]{h}^{\theta}\norm[\dot{H}^s]{\nabla h}^{1-\theta}
		\end{align*}
		with $ \theta\in(0,1) $ and apply Young's inequality yielding
		\begin{align*}
			-\dualbra{\Lin g^i}{g^i} \leq -c_\delta \norm[H^{s,*}_q]{g^i}^2 + \dfrac{c_\delta}{2}\norm[H^{s,*}_q]{\nabla h}^2+C_\delta''\norm[L^1]{h}^2.
		\end{align*}
	\end{enumerate}
	We now sum the estimates for $ i=1,2,3 $ to get
	\begin{align*}
		\dfrac{1}{2}\dfrac{d}{dt}\norm[L^2_q]{\nabla h}^2 \leq& (C_4+C_7)\left( \norm[L^2_q]{\nabla h}^2 + \norm[L^2_p]{h}\norm[L^2_q]{\nabla h} \right)  
		\\&+ \eta_tC_5 \left( \norm[H^s_{p+\gamma/2}]{h}\norm[H^s_{q+\gamma/2}]{\nabla h} + \norm[L^2_{p+\gamma/2}]{h}\norm[L^2_{q+\gamma/2}]{\nabla h} \right) 
		\\
		&+ \eta_t C_8 \norm[L^2_q]{\nabla h}\norm[H^{s,*}_p]{h}\norm[H^{s,*}_q]{\nabla h}
		\\
		&-\left( \dfrac{\eta_t\, c_\delta}{2}-C_6-\eta_tC_8\norm[L^2_p]{h} \right)\norm[H^{s,*}_q]{\nabla h}^2 + \dfrac{C_4}{\eta_t}\norm[L^2]{\nabla h}+ \eta_tC_\delta''\norm[L^1]{h}^2.
	\end{align*}
	For the first line we use rough estimates in $ H^{s,*} $ and for the third line we use Young's inequality. For the second line we use Young's inequality to absorb both norms with $ \nabla h $ in $ c_\delta $. We thus obtain
	\begin{align}\label{eq:ProofH1Est}
		\begin{split}
			\dfrac{1}{2}\dfrac{d}{dt}\norm[L^2_q]{\nabla h}^2 \leq& \left(C_4+C_7+\eta_tC_5'+\eta_tC_8\norm[L^2_q]{\nabla h}\right)\norm[H^{s,*}_p]{h}^2
			\\
			&-\left( \dfrac{\eta_t\, c_\delta}{4}-C_4-C_6-C_7-\eta_tC_8\norm[L^2_p]{h}-\eta_tC_8\norm[L^2_q]{\nabla h} \right)\norm[H^{s,*}_q]{\nabla h}^2 
			\\
			&+\dfrac{C_4}{\eta_t}\norm[L^2]{\nabla h}+ \eta_tC_\delta''\norm[L^1]{h}^2.
		\end{split}
	\end{align}
	Next, due to \eqref{eq:LoopInequality1}, i.e. $ \eta_t\geq c_0\bar{\eta}(\beta_0) $ by \eqref{eq:GrowthCollisionConstant}, and $ \bar{\eta}(\beta_0)\to \infty $ as $ \beta_0\to 0 $, we can find $ \varepsilon_0' $ such that the constants $ C_2 $ resp. $ C_4,\, C_6,\, C_7 $ in \eqref{eq:ProofL2Est} resp. \eqref{eq:ProofH1Est} can be absorbed in the term involving $ c_\delta $. We then obtain
	\begin{align*}
		\dfrac{1}{2}\dfrac{d}{dt}\tnorm[\mathcal{H}^1_p]{h}^2 \leq& -\left( \dfrac{\eta_t\, c_\delta}{4} -\eta_t \, C_3\norm[L^2_p]{h}-\kappa\eta_tC_5'-\eta_tC_8\, \kappa\norm[L^2_q]{\nabla h}\right) \norm[H^{s,*}_p]{h}^2 
		\\
		&-\left( \dfrac{\eta_t\, c_\delta}{8}-\eta_tC_8\norm[L^2_p]{h} -\eta_tC_8\norm[L^2_q]{\nabla h} \right)\, \kappa\norm[H^{s,*}_q]{\nabla h}^2
		\\
		&+\dfrac{C_1+C_4}{\eta_t}\tnorm[\mathcal{H}^1_p]{h}+ \eta_t\left( C_\delta'+ \kappa C_\delta'' \right) \norm[L^1]{h}^2.
	\end{align*}
	Now, let us choose $ \kappa\in(0,1) $ small enough such that $ C_5' $ can be absorbed into the term involving $ c_\delta $, yielding (recall that $ q=p-2s $)
	\begin{align}\label{eq:ProofH1pEst}
		\begin{split}
			\dfrac{1}{2}\dfrac{d}{dt}\tnorm[\mathcal{H}^1_p]{h}^2 &\leq -\eta_t \left( \dfrac{c_\delta}{8}-C_9\tnorm[\mathcal{H}^1_p]{h} \right) \left( \norm[H^{s,*}_p]{h}^2 +\kappa\norm[H^{s,*}_{p-2s}]{\nabla h}^2 \right) 
			\\
			&\dfrac{C_9}{\eta_t}\tnorm[\mathcal{H}^1_p]{h}+ \eta_tC_\delta''' \norm[L^1]{h}^2.
		\end{split}
	\end{align}
	The constants depend on $ \kappa $, which is a fixed numerical constant.
	
	\textit{Step 2:} Let us derive an a priori bound on $ \tnorm[\mathcal{H}^1_p]{h_t} $, which is used in the next step for a continuation argument. For this let us assume that
	\begin{align}\label{eq:ProofAssumpAPrioriEst}
		\tnorm[\mathcal{H}^1_p]{h_t} \leq \Omega\left( \dfrac{\varepsilon}{(1+t)^2}+\dfrac{1}{\eta_t^2} \right) \leq \dfrac{c_\delta}{16C_9}
	\end{align}
	holds on $ [0,T] $ for some $ \Omega>0 $. As a consequence of \eqref{eq:ProofH1pEst} we have
	\begin{align*}
		\dfrac{1}{2}\dfrac{d}{dt}\tnorm[\mathcal{H}^1_p]{h}^2 \leq -\eta_t c_\delta' \tnorm[\mathcal{H}^1_p]{h}^2+ \dfrac{C_9}{\eta_t}\norm[\mathcal{H}^1_p]{h}+ \eta_tC_\delta'''\norm[L^1]{h}^2.
	\end{align*}
	We apply Gronwall's inequality, \eqref{eq:PropAssL1} and \eqref{eq:ProofAssumpAPrioriEst} to get
	\begin{align}\label{eq:ProofAPrioriEstPrep}
		\begin{split}
			\tnorm[\mathcal{H}^1_p]{h_t}^2 &\leq e^{-E_t}\varepsilon^2 + C_9\Omega\,\int_0^te^{-E_t+E_s}\left( \dfrac{\varepsilon}{\eta_s(1+s)^2}+\dfrac{1}{\eta_s^3} \right) \,ds
			\\
			&+ 2C_\delta'''\Omega'\, \int_0^te^{-E_t+E_s}\left( \dfrac{\varepsilon^2\, \eta_s}{(1+s)^4}+\dfrac{1}{\eta_s^3}\right) \,ds,
		\end{split}
	\end{align}
	where $ E_t:=c_\delta'\int_0^t\eta_s\, ds $. By \eqref{eq:LoopInequality1} and hence \eqref{eq:GrowthCollisionConstant} we have
	\begin{align*}
		e^{-E_t} \leq \dfrac{C_{10}}{(1+t)^4}.
	\end{align*}
	Note that this does not depend on $ \beta_0 $, since we can assume $ \bar{\eta}(\beta_0)\geq1 $, due to the smallness $ \beta_0\leq \varepsilon_0' $. Let us now estimate the time integrals. 
	
	\begin{enumerate}[(i)]
		\item We use partial integration to get
		\begin{align*}
			\int_0^te^{-E_t+E_s}\dfrac{1}{\eta_s^3}\, ds \leq \dfrac{1}{c_\delta'\eta_t^4}+\dfrac{4}{c_\delta'}\int_0^te^{-E_t+E_s}\dfrac{\eta_s'}{\eta_s^5}\, ds.
		\end{align*}
		We now use that $ |\eta_t'/\eta_t|\leq C $ as in \eqref{eq:ProofTimeIntegralBound} and $ \eta_t\geq c_0\bar{\eta} $, cf. \eqref{eq:GrowthCollisionConstant}, to get
		\begin{align*}
			\int_0^te^{-E_t+E_s}\dfrac{1}{\eta_s^3}\, ds \leq \dfrac{1}{c_\delta'\,\eta_t^4}+\dfrac{1}{c_0c_\delta'\,\bar{\eta}}\int_0^te^{-E_t+E_s}\dfrac{1}{\eta_s^3}\, ds.
		\end{align*}
		If $ \beta_0\leq \varepsilon_0' $ is small enough, i.e. $ \bar{\eta}(\beta_0) $ large enough, then we can absorb the last term into the left-hand side yielding
		\begin{align*}
			\int_0^te^{-E_t+E_s}\dfrac{1}{\eta_s^3}\, ds \leq \dfrac{C_{11}}{\eta_t^4}.
		\end{align*}
		Note that in this argument the smallness of $ \varepsilon_0' $ depends only on numerical constants.
		
		\item We again use partial integration to get
		\begin{align*}
			\int_0^te^{-E_t+E_s}\dfrac{1}{\eta_s(1+s)^2}\, ds \lesssim \dfrac{1}{\eta_t^2(1+t)^2}+\int_0^te^{-E_t+E_s}\left( \dfrac{1}{\eta_s^2(1+s)^2}+\dfrac{1}{\eta_s^2(1+s)^3}\right) \, ds,
		\end{align*}
		where we used again $ |\eta_t'/\eta_t|\leq C $. As before we can absorb the last term into the left-hand side for $ \bar{\eta}(\beta_0) $ large enough, $ \beta_0\leq \varepsilon_0' $. Hence, we have
		\begin{align*}
			\int_0^te^{-E_t+E_s}\dfrac{1}{\eta_s(1+s)^2}\, ds \leq \dfrac{C_{12}}{\eta_t^2(1+t)^2}.
		\end{align*}
		We then use in the corresponding term in \eqref{eq:ProofAPrioriEstPrep}
		\begin{align*}
			\dfrac{\varepsilon}{\eta_t^2(1+t)^2}\leq \dfrac{1}{\eta_t^4}+\dfrac{\varepsilon^2}{(1+t)^4}.
		\end{align*}
		
		\item Finally, have
		\begin{align*}
			\int_0^te^{-E_t+E_s}\eta_s\dfrac{1}{(1+s)^4}\,ds \lesssim \dfrac{1}{(1+t)^4} + \int_0^te^{-E_t+E_s}\dfrac{1}{(1+s)^5}\,ds.
		\end{align*}
		We estimate the last integral by considering the case $ t\leq 1 $ and $ t\geq1 $. In the second case, we split $ [0,t] $ into $ [0,t/2] $ and $ [t/2,t] $. This yields with \eqref{eq:LoopInequality1} and hence \eqref{eq:GrowthCollisionConstant}
		\begin{align*}
			\int_0^te^{-E_t+E_s}\dfrac{1}{(1+s)^5}\, ds &\leq  \ind_{\set{t\leq 1}}\dfrac{C}{(1+t)^4} + \ind_{\set{t\geq 1}}\left( e^{-c_\delta'\int_{t/2}^t \eta_r\, dr} + \dfrac{C}{(1+t)^4} \right)
			\\
			&\leq \dfrac{C_{13}}{(1+t)^4}.
		\end{align*}
		Note that $ C_{13} $ does not depend on $ \beta_0 $, since $ \bar{\eta}(\beta_0)\geq1 $.
	\end{enumerate}
	Hence, we obtain the estimate
	\begin{align*}
		\tnorm[\mathcal{H}^1_p]{h_t}^2 \leq  C_{14}\left(\Omega'+\sqrt{\Omega}+1\right)  \left( \dfrac{\varepsilon}{(1+t)^2}+\dfrac{1}{\eta_t^2}\right) .
	\end{align*}
	for some constant $ C_{14}\geq 1 $. If we would have $ \Omega=16\,C_{14}^2(\Omega'+1) $ then we would obtain
	\begin{align}\label{eq:ProofAPrioriEst2}
		\tnorm[\mathcal{H}^1_p]{h_t}^2 \leq  \dfrac{\Omega}{2}\left( \dfrac{\varepsilon}{(1+t)^2}+\dfrac{1}{\eta_t^2}\right) .
	\end{align}
	Recall that for this to hold we needed $ \tnorm[\mathcal{H}^1_p]{h_t} \leq c_\delta/16C_9 $ and $ \beta_0\leq \varepsilon_0' $.

	\textit{Step 3: Continuation argument.} We now use a continuation argument to show that the a priori estimates in Step 2 can be justified rigorously, i.e. it holds
	\begin{align}\label{eq:ProofContArgEst}
		\tnorm[\mathcal{H}^1_{p}]{h_t} \leq \Omega\left( \dfrac{\varepsilon}{(1+t)^2}+\dfrac{1}{\eta_t^2}\right) .
	\end{align}
	We define the constant $ \Omega:=16\,C_{14}^2(\Omega'+1)\geq2 $ with the previous notation. First of all, the estimate \eqref{eq:ProofContArgEst} is true on some small interval $ [0,t_0] $, $ t_0>0 $, due to $ \tnorm[\mathcal{H}^1_{p}]{h_0}\leq \varepsilon $, $ \Omega\geq2 $ and the continuity of the norm. Let us now assume it holds on some interval $ [0,t_1] $. Without loss of generality this interval is closed by the continuity of the norm. If we further reduce $ \varepsilon_0'=\varepsilon_0'(\Omega')>0 $ such that 
	\begin{align*}
		\Omega\left( \dfrac{\varepsilon}{(1+t)^2}+\dfrac{1}{\eta_t^2}\right) \leq \dfrac{c_\delta}{16C_9}
	\end{align*}
	for $ \varepsilon\leq \varepsilon_0' $ and $ \beta_0\leq \varepsilon_0' $, the estimate \eqref{eq:ProofAPrioriEst2} is also true on $ [0,t_1] $ by Step 2. By continuity the estimate \eqref{eq:ProofContArgEst} is then also valid on some larger interval $ [0,t_2] $, $ t_2>t_1 $. This shows that the set of times for which \eqref{eq:ProofContArgEst} is valid, is both open and closed. Hence, \eqref{eq:ProofContArgEst} holds for all $ t\geq0 $. Note that the bound \eqref{eq:ProofContArgEst} implies \eqref{eq:PropH1Est} up to the numerical factor $ \kappa\in(0,1) $. This concludes the proof with $ C'=16\,C_{14}^2 $.
\end{proof}

\paragraph{Estimates in $ L^1 $-framework.} Based on the estimates in Proposition \ref{pro:PrepLoopInequality2} we prove that \eqref{eq:LoopInequality1} implies \eqref{eq:LoopInequality2}. For this we need the following result from \cite{Tristani2014ExponentialConvergenceEquilHomogBoltz} for the linearized collision operator.
\begin{lem}\label{lem:LinearDecay} 
	For any $ m>2 $ the semigroup generated by $ -\Lin $, $ e^{-t \Lin}:L^1_m\to L^1_m $, has the following property: there is $ C_m, \lambda_m>0 $ such that for all $ t\geq0 $
	\begin{align*}
		\norm[L^1(\left\langle  v \right\rangle^m\, dv)]{e^{-t \Lin} g-\Pi_0g}\leq C_me^{-\lambda_mt}\norm[L^1_m]{g}.
	\end{align*}
	Here, $ \Pi_0 $ denotes the projection onto $ \ker\Lin $ in $ L^1_m $.
\end{lem}
The next lemma gives an estimate of the collision operator in $ L^1 $. It can be proved in a similar way as \cite[Proposition 3.1]{Tristani2014ExponentialConvergenceEquilHomogBoltz}.
\begin{lem}\label{lem:L1EstimateCollOp}
	For any two functions $ f,\, g $ we have
	\begin{align*}
		\norm[L^1_m]{Q(f,g)}\leq C\left( \norm[L^1_{m+\gamma}]{f}\norm[L^1_{m+\gamma}]{g}+\norm[L^1_{\gamma+1}]{f}\norm[W^{1,1}_{m+\gamma+1}]{g} \right).
	\end{align*}
\end{lem}
\begin{pro}\label{pro:LoopInequality2}
	Under the assumptions of Theorem \ref{thm:MainTheorem} there are a constant $ \Omega $ and a sufficiently small constant $ \varepsilon_0''\in(0,1) $ such that the following holds. Assuming that \eqref{eq:LoopInequality1} is true on some interval $ [0,T] $ and $ \varepsilon\leq \varepsilon_0'' $ and $ \beta_0\leq \varepsilon_0'' $ we have for all $ t\in [0,T] $
	\begin{align}\label{eq:PropLoopInequality2}
		\norm[\mathcal{H}^1_{p_0}]{h_t}\leq \Omega\left( \dfrac{\varepsilon}{(1+t)^2}+\dfrac{1}{\eta_t^2}\right) .
	\end{align}
\end{pro}

\begin{proof}
	We use Duhamel's formula together with the properties of the semigroup $ P_{0,t} $ generated by $ \eta_t\Lin $ to obtain from \eqref{eq:ErrorEq}
	\begin{align}\label{eq:ProofDuhamelEstimate}
		\begin{split}
			\norm[L^1_m]{h_t}\leq \norm[L^1_m]{P_{0,t}h_0} + \int_{0}^{t} &\bigg[ \norm[L^1_m]{P_{r,t}(S_r+(\mathscr{R}h)_r-\alpha_r^2\div(vh_r))}
			\\
			&+ \eta_r\norm[L^1_m]{P_{r,t}Q(h_r,h_r)} \bigg] dr.
		\end{split}
	\end{align}
	
	\textit{Step 1:} Let us first estimate \eqref{eq:ProofDuhamelEstimate}. As can be checked by the time-change $ \tau(t)=\int_0^t\eta_s\,ds $, Lemma~\ref{lem:LinearDecay} yields
	\begin{align*}
		\norm[L^1_m]{P_{s,t}g-\Pi_0g}\leq C_me^{-E_t+E_s}\norm[L^1_m]{g},
	\end{align*}
	where $ E_t:=\lambda_m \int_{0}^{t}\eta_sds $ and $ \lambda_m>0 $ is defined by $ \Lin $, see Lemma \ref{lem:LinearDecay}. Since $ S_r+(\mathscr{R}h)_r-\alpha_r^2\div(vh_r) $ and $ Q(h_r,h_r) $ are in $ (\ker\Lin)^\perp $ we obtain
	\begin{align*}
		\norm[L^1_m]{h_t}\leq e^{-E_t}\norm[L^1_m]{h_0} +\int_0^t e^{-E_t+E_r} &\bigg[ \norm[L^1_m]{S_r}+\norm[L^1_m]{(\mathscr{R}h)_r}+\left| \alpha_r^2 \right| \norm[L^1_m]{\div(vh_r)} 
		\\
		&+ \eta_r\norm[L^1_m]{Q(h_r,h_r)} \bigg]\, dr. 
	\end{align*}
	We recall that with $ m>2 $
	\begin{align}\label{eq:ProofLagrangeMultEstimate}
		|\alpha_t^1|\leq \dfrac{C}{\eta_t}, \quad |\alpha_t^2|\leq C\norm[L^1_m]{h_t}.
	\end{align}
	Let us now estimate term by term. 
	\begin{enumerate}[(i)]
		\item We obtain from Lemma \ref{lem:FirstOrderEstimate} and Lemma \ref{lem:L1EstimateCollOp}
		\begin{align*}
		\norm[L^1_m]{S}\lesssim  \dfrac{1}{\eta_t}+\dfrac{1}{\eta_t^2}\lesssim \dfrac{1}{\eta_t}.
		\end{align*}
		
		\item For the term $ \mathscr{R}h $ we obtain similarly
		\begin{align*}
			\norm[L^1_m]{\mathscr{R}h}&\lesssim \left( |\alpha_t^2| + \norm[W^{1,1}_{m+1}]{h_t} + |\alpha_t^1|\norm[W^{1,1}_{m+1}]{h_t} + \eta_t \norm[W^{1,1}_{\gamma+m+1}]{\bar{\mu}_t}\norm[W^{1,1}_{\gamma+m+1}]{h_t} \right)
			\\
			&\lesssim \norm[W^{1,1}_{\gamma+m+1}]{h_t} \lesssim \norm[H^1_{\gamma+m+1+2s+3/2}]{h_t} \lesssim  \norm[\mathcal{H}^1_{p_0}]{h_t},
		\end{align*}
		where we used $ p_0-2s\geq m+2+2s+3/2\geq m+\gamma+1+2s+3/2 $, $ \gamma\leq1 $. 
		
		\item In addition, we have with \eqref{eq:ProofLagrangeMultEstimate}
		\begin{align*}
			|\alpha_t^2|\norm[L^1_m]{\div(vh_t)} \lesssim \norm[W^{1,1}_{m+1}]{h_t}\norm[L^1_m]{h_t}\lesssim \norm[\mathcal{H}^1_{p_0}]{h_t}\norm[L^1_m]{h_t}\lesssim \norm[\mathcal{H}^1_{p_0}]{h_t}^2.
		\end{align*}
		
		\item Finally, we apply Lemma \ref{lem:L1EstimateCollOp} to get
		\begin{align*}
			\norm[L^1_m]{Q(h,h)}\lesssim \left( \norm[L^1_{m+\gamma}]{h}^2+\norm[L^1_{\gamma+1}]{h}\norm[W^{1,1}_{m+\gamma+1}]{h} \right) \lesssim \norm[\mathcal{H}^1_{p_0}]{h_t}^2.
		\end{align*}
	\end{enumerate}
	Putting all bounds together yields
	\begin{align}\label{eq:ProofDuhamelEstimate2}
		\norm[L^1_m]{h_t}\leq C_me^{-E_t}\norm[L^1_m]{h_0} + C\int_0^te^{-E_t+E_s} \left( \dfrac{1}{\eta_s} + \norm[\mathcal{H}^1_{p_0}]{h_s} + \norm[\mathcal{H}^1_{p_0}]{h_t}^2+\eta_s\norm[\mathcal{H}^1_{p_0}]{h_s}^2 \right) \, ds.
	\end{align}
	
	\textit{Step 2:} Let us now give an a priori estimate assuming
	\begin{align}\label{eq:ProofProp2APrioriEst}
		\norm[L^1_m]{h_t}\leq \Omega'\left( \dfrac{\varepsilon}{(1+t)^2}+\dfrac{1}{\eta_t^2}\right), \quad \norm[\mathcal{H}^1_{p_0}]{h_t}\leq C'(\Omega'+1)\left( \dfrac{\varepsilon}{(1+t)^2}+\dfrac{1}{\eta_t^2}\right).
	\end{align}
	Here, $ C' $ is the constant determined in Proposition \ref{pro:PrepLoopInequality2} in \eqref{eq:PropH1Est}. Let us define for brevity $ \Omega:=C'(\Omega'+1) $.
	We estimate the time integrals in \eqref{eq:ProofDuhamelEstimate2} using \eqref{eq:ProofProp2APrioriEst} and obtain
	\begin{align}\label{eq:ProofDuhamelEstimate3}
		\norm[L^1_m]{h_t}&\leq C_me^{-E_t}\varepsilon + C\int_0^te^{-E_t+E_s} \left( \dfrac{1}{\eta_s} + \dfrac{\Omega\varepsilon}{(1+s)^2} + \dfrac{\Omega}{\eta_s^2}+\dfrac{\eta_s\Omega^2\varepsilon^2}{(1+s)^4}+\dfrac{\Omega^2}{\eta_s^3} \right) \, ds.
	\end{align}
	Let us estimate term by term.
	\begin{enumerate}[(i)]
		\item As in the proof of Proposition \ref{pro:PrepLoopInequality2} (Step 2, (i)) we can use partial integration and $ \bar{\eta}(\beta_0) $ large enough to get
		\begin{align*}
			\int_0^te^{-E_t+E_s}\dfrac{1}{\eta_s} \, ds \leq\dfrac{C_1}{\eta_t^2}.
		\end{align*}
		
		\item For the second and third term in \eqref{eq:ProofDuhamelEstimate3} we estimate
		\begin{align*}
			\int_0^te^{-E_t+E_s}\left( \dfrac{\varepsilon}{(1+s)^2}+\dfrac{1}{\eta_s^2} \right)  \, ds \leq \dfrac{C_2\varepsilon}{\eta_t(1+t)^2}+\dfrac{C_2}{\eta_t^3}+C_2\int_0^te^{-E_t+E_s}\left( \dfrac{\varepsilon}{\eta_s(1+s)^2}+\dfrac{1}{\eta_s^3} \right)  \, ds.
		\end{align*}
		For $ c_0\bar{\eta}(\beta_0)\leq \eta_t $ large enough, i.e. $ \beta_0\leq \varepsilon_0'' $ small enough, we can absorb the last term into the left-hand side.
	
		\item Finally, we treat the last two terms in \eqref{eq:ProofDuhamelEstimate3} at once. We obtain as in the proof of Proposition~\ref{pro:PrepLoopInequality2} in Step 2, (i) and (ii)
		\begin{align*}
			\int_0^te^{-E_t+E_s}\left( \dfrac{\varepsilon^2\, \eta_s}{(1+s)^4}+\dfrac{1}{\eta_s^3} \right) \, ds \leq \dfrac{C_3\varepsilon^2}{(1+t)^4}+\dfrac{C_3}{\eta_t^4}.
		\end{align*}
	\end{enumerate}
	Combining the above estimates leads to
	\begin{align*}
		\norm[L^1_m]{h_t}\leq \left[ C_m+C_1+\dfrac{C_2\Omega}{\bar{\eta}}+C_3\Omega^2\left( \varepsilon+\dfrac{1}{\bar{\eta}} \right) \right]\left( \dfrac{\varepsilon}{(1+t)^2}+\dfrac{1}{\eta_t^2} \right).
	\end{align*}
	By the definition of $ \Omega=C'(\Omega'+1) $ we get for some $ C'' $
	\begin{align*}
		\norm[L^1_m]{h_t}\leq C''\,\left[ 1+\dfrac{\Omega'}{\bar{\eta}}+(\Omega')^2\left( \varepsilon+\dfrac{1}{\bar{\eta}} \right) \right]\left( \dfrac{\varepsilon}{(1+t)^2}+\dfrac{1}{\eta_t^2} \right).
	\end{align*}
	If we would have $ \Omega'= 6\max\left\lbrace C'',1\right\rbrace $ and  $ \varepsilon $, $ \beta_0\leq \varepsilon_0'' $ are sufficiently small, such that
	\begin{align}\label{eq:ProofConditionConstants}
		\dfrac{\Omega'}{\bar{\eta}(\beta_0)} \leq 1, \quad (\Omega')^2\left( \varepsilon+\dfrac{1}{\bar{\eta}(\beta_0)} \right) \leq 1
	\end{align}
	holds, then we would obtain
	\begin{align}\label{eq:ProofAPrioriL1Est}
		\norm[L^1_m]{h_t}\leq \dfrac{\Omega'}{2}\left( \dfrac{\varepsilon}{(1+t)^2}+\dfrac{1}{\eta_t^2}\right).
	\end{align}
	In the next step, we show that this particular choice enables us to conclude the proof.

	\textit{Step 3: Continuation argument.} We now prove that \eqref{eq:ProofProp2APrioriEst} holds using a continuation argument. Here, we define $ \Omega'= 6\max\left\lbrace C'',1\right\rbrace $ as in the end of the last step. We can assume without loss of generality that $ \Omega'> C_* $, where $ C_* $ satisfies
	\begin{align*}
		\norm[L^1_m]{g}\leq C_*\norm[\mathcal{H}^1_{p_0}]{g}.
	\end{align*} 
	Let us also recall that $ C' $ is the constant in Proposition \ref{pro:PrepLoopInequality2}. Furthermore, we choose $ \varepsilon_0''\leq \varepsilon_0' $ such that \eqref{eq:ProofConditionConstants} is valid for $ \varepsilon\leq \varepsilon_0'' $, $ \beta_0\leq \varepsilon_0'' $. Note that here $ \varepsilon_0' $ is the constant in Proposition \ref{pro:PrepLoopInequality2}, which depends our choice of $ \Omega' $. However, $ \Omega' $ is now fixed. Let us now proceed with the continuation argument.
	
	First, the estimates \eqref{eq:ProofProp2APrioriEst} are true on some small interval $ [0,t_0] $, since by assumption $ \norm[\mathcal{H}^1_{p_0}]{h_0}\leq \varepsilon $ and $ \norm[L^1_{m}]{h_0}\leq C_*\varepsilon $. Let us now assume that \eqref{eq:ProofProp2APrioriEst} hold on some interval $ [0,t_1] $. We want to extend it by continuity to some larger interval. For $ t\in [0,t_1] $ we obtain from the previous step that \eqref{eq:ProofAPrioriL1Est} is valid. Hence, we can extend this bound on some larger interval $ [0,t_2] $, $ t_2>t_1 $. Using now Proposition \ref{pro:PrepLoopInequality2} (noting that all the assumptions are satisfied) on the interval $ [0,t_2] $ we get the second estimate in \eqref{eq:ProofProp2APrioriEst}. As a consequence \eqref{eq:ProofProp2APrioriEst} is valid on the whole interval $ [0,T] $. This yields \eqref{eq:PropLoopInequality2} by defining the numerical constant $ \Omega:=C'(\Omega'+1) $ and concludes the proof.
\end{proof}

\subsubsection{Conclusion of proof}
With this preparation we can give the proof of Theorem \ref{thm:MainTheorem}.
\begin{proof}[Proof of Theorem \ref{thm:MainTheorem}]
	We select the constant $ \Omega $ as well as $ \varepsilon_0 $ to ensure that \eqref{eq:LoopInequality1}, \eqref{eq:LoopInequality2} hold for all times. We make the following choices.
	\begin{enumerate}[(i)]
		\item Define $ \Omega $ as in Proposition \ref{pro:LoopInequality2}.
		
		\item Select $ \varepsilon_0\in(0,1) $ such that $ \varepsilon_0\leq \varepsilon_0'' $, the constant $ \varepsilon_0'' $ given in Proposition \ref{pro:LoopInequality2}, and 
		\begin{align}\label{eq:ProofEpsilonCond}
			\exp\left( c\gamma\, \Omega\,\int_0^\infty\left( \dfrac{\varepsilon}{(1+t)^2}+\dfrac{1}{c_0^2(\bar{\eta}(\beta_0)+t)^2} \right) \, dt \right)\leq 2
		\end{align}
		holds for all $ \varepsilon,\, \beta_0\leq \varepsilon_0 $. The integral on the left-hand side is motivated by the term $ R_T(\varepsilon,\beta_0) $ in Lemma \ref{lem:LoopInequality1}. Using formally $ \eta_t\geq c_0(\bar{\eta}+t) $, according to \eqref{eq:GrowthCollisionConstant}, gives the above integral for $ T=\infty $.
	\end{enumerate}
	Finally, we use again a continuation argument to prove \eqref{eq:LoopInequality1} and \eqref{eq:LoopInequality2}. First of all, by continuity the estimate \eqref{eq:LoopInequality1} holds on some small interval $ [0,t_0] $, $ t_0>0 $, since $ \eta_0=\nu_0\beta_0^{-\gamma/2}=Z_0(\beta_0) $. Recall the definition of $ Z_t(\beta_0) $ in \eqref{eq:GrowthInversTemp}. The assumptions of Proposition \ref{pro:LoopInequality2} are satisfied on $ [0,t_0] $, so that \eqref{eq:LoopInequality2} is valid on $ [0,t_0] $ as well.
	
	Let us assume that \eqref{eq:LoopInequality1} and \eqref{eq:LoopInequality2} hold on some interval $ [0,t_1] $. This interval can be assumed to be closed by continuity. Lemma \ref{lem:LoopInequality1} yields for $ t\in[0,t_1] $
	\begin{align*}
		\exp\left( -c\gamma \, 
		\Omega_2\, R_{t_1}(\varepsilon,\beta_0)\right)\, Z_t(\beta_0)\leq \eta_t \leq \exp\left( c\gamma \, 
		\Omega_2\, R_{t_1}(\varepsilon,\beta_0)\right)\, Z_t(\beta_0).
	\end{align*} 
	Using $ \eta_t\geq c_0(\bar{\eta}+t) $ for $ t\in[0,t_1] $, we get with \eqref{eq:ProofEpsilonCond}
	\begin{align*}
		\dfrac{1}{2}Z_t(\beta_0)\leq \eta_t \leq 2Z_t(\beta_0).
	\end{align*}
	Hence, we can extend \eqref{eq:LoopInequality1} on some larger interval $ [0,t_2] $, $ t_2>t_1 $. On this interval we can apply Proposition \ref{pro:LoopInequality2}, which yields also \eqref{eq:LoopInequality2} on $ [0,t_2] $. Thus, both \eqref{eq:LoopInequality1} and \eqref{eq:LoopInequality2} hold for all times. These estimates imply \eqref{eq:SecondOrderHilbertExpDecay} and \eqref{eq:InvTempAsyBehav}. Finally, let us note that \eqref{eq:FirstOrderHilbertExpDecay} is a consequence of Lemma~\ref{lem:FirstOrderEstimate}.
\end{proof}

\section{Application to homoenergetic solutions}
\label{sec:Application}
In this section, we apply Theorem \ref{thm:MainTheorem} to homoenergetic solutions in the case of \textit{simple shear}, \textit{simple shear with decaying planar dilatation/shear} and \textit{combined orthogonal shear} to conclude Theorem \ref{thm:HomoenergticSol1} and Theorem \ref{thm:HomoenergticSol2}. To this end, let us first give a lemma that allows to verify assumption $ (II) $ in Theorem \ref{thm:MainTheorem}.
\begin{lem}\label{lem:AssumptionII}
	Let $ L \in C^1([0,\infty);\R^{3\times3}) $ and $ \nu\in C^1([0,\infty);(0,\infty)) $ satisfy assumption $ (I) $ in Theorem \ref{thm:MainTheorem}. Consider the decomposition of $ L_t $ into its trace-free and trace part $ L_t= A_t+ b_t\, I $. Assume that
	\begin{enumerate}[(i)]
		\item $ A_t = A^0 + A^1_t $ with $ \trace A^0=0 $, $ A^0\neq 0 $ and $ A^1_t\to 0 $ as $ t\to \infty $;
		\item $ b_t = b_t^0+b^1_t $ with $ b_t^0\geq0 $ and $ |b^1_t|\leq C/(1+t)^2 $ for all $ t\geq0 $ with some constant $ C>0 $;
		\item $ t\mapsto \nu_t $ is bounded on $ [0,\infty) $;
		\item we have for all $ t\geq 1 $
		\begin{align*}
			N_t:=\int_0^t \dfrac{\nu_t}{\nu_s}e^{-\int_s^t b_r\, dr}\, ds \approx t.
		\end{align*}
	\end{enumerate}
	Then, assumption $ (II) $ in Theorem \ref{thm:MainTheorem} is satisfied.
\end{lem}
Note that for the matrix $ L $ given by \eqref{eq:SimpleShear}, \eqref{eq:SimpShearDecPlanarDil} or \eqref{eq:CombOrthShear} we have $ \trace L_t= 3b_t\geq 0 $ up to terms of order $ \mathcal{O}(1/t^2) $. This motivates assumption $ (ii) $ in the lemma.
\begin{proof}[Proof of Lemma \ref{lem:AssumptionII}]
	First of all, let us note that 
	\begin{align*}
		a_t=\dualbra{v\cdot A_tv \, \mu}{\Lin^{-1}\left[ v\cdot A_tv \, \mu \right] }_{L^2(\mu^{-1/2})} \to \dualbra{v\cdot A^0v \, \mu}{ \Lin^{-1}\left[ v\cdot A^0v \, \mu \right] }_{L^2(\mu^{-1/2})}>0.
	\end{align*}
	Recall that $ v\cdot A_tv \, \mu\in(\ker\Lin)^\perp $ and that $ \Lin $ is a positive operator on $ (\ker\Lin)^\perp $. Since $ t\mapsto a_t>0 $ is continuous, we have $ 0<c_0\leq a_t \leq C_0 $ for all $ t\geq0 $ and some constants $ c_0,\, C_0>0 $.
	
	Due to assumption $ (ii) $ and $ (iii) $ the first term in $ Z_t(1) $ in \eqref{eq:GrowthInversTemp} is bounded. With the above observation, the second term is equivalent to 
	\begin{align*}
		\int_0^t \dfrac{\nu_t}{\nu_s}e^{-\int_s^t b_r\, dr}\, ds \approx t.
	\end{align*}
	This implies $ Z_t(1)\approx 1+t $ for $ t\geq0 $.
\end{proof}
	Let us now give the proof of Theorem \ref{thm:HomoenergticSol1} and Theorem \ref{thm:HomoenergticSol2}.
	\begin{proof}[Proof of Theorem \ref{thm:HomoenergticSol1} and Theorem \ref{thm:HomoenergticSol2}] First of all, let us recall that a solution $ g $ to \eqref{eq:homenergBE} is related to a solution $ f $ to, cf.\eqref{eq:IndtrodModelHomoengBE},
	\begin{align}\label{eq:IndtrodModelHomoengBEProof}
		\begin{split}
			\partial_t f &= \div \left( \left( L_t-\alpha_t\right) v \, f\right)  + \rho_t\, \beta_t^{-\gamma/2}\, Q(f,f), \quad f(0,\cdot)=f_0(\cdot),
			\\
			\beta_t &= \beta_0\exp \left( 2\int_{0}^t\alpha_s\, ds \right),
			\quad
			\alpha_t:=\dfrac{1}{3}\int v\cdot L_tv \, f_t(v)\,dv,
			\\
			\rho_t &= \exp\left( -\int_0^t \trace L_s \, ds \right)
		\end{split}
	\end{align}	
	via the scaling $ f_t(v)=g_t(v\beta_t^{-1/2}+V_t)\beta_t^{-3/2}\rho_t^{-1} $, cf. Theorem \ref{thm:HomoenergticSol1}. Also recall that with this scaling $ f $ satisfies the normalization \eqref{eq:Normalization}. As already indicated earlier we use equation \eqref{eq:ModelEq} for a specific choice of $ \nu $ to draw conclusions for solutions to \eqref{eq:IndtrodModelHomoengBEProof}. We discuss this reduction in each case of \textit{simple shear}, \textit{simple shear with decaying planar dilatation/shear} and \textit{combined orthogonal shear} separately. Let us recall that the inverse temperature $ \beta_t $ in \eqref{eq:ModelEq} satisfies the following equation
	\begin{align}\label{eq:ProofEqInverseTemp}
		\dfrac{\beta_t'}{2\beta_t} = \alpha_t = \dfrac{1}{3}\int v\cdot L_tv \,  f_t\, dv.
	\end{align}

	\textit{Simple Shear:} The matrix $ L_t $ is given by \eqref{eq:SimpleShear}, in particular it is constant in time and $ L=A $ is trace-free. This implies that the density is constant $ \rho_t=\rho_0 $. We set $ \nu_t\equiv \rho_0 $ and equation \eqref{eq:IndtrodModelHomoengBEProof} reduces to \eqref{eq:ModelEq}. Let us now check the structural conditions of Theorem \ref{thm:MainTheorem}. Assumption $ (I) $ is satisfied. Furthermore, Lemma \ref{lem:AssumptionII} applies with $ N_t=t $, yielding assumption $ (II) $. Note that the definition of $ \bar{\mu} $ in \eqref{eq:HilbertExpFirstOrder3} is the same as in Theorem \ref{thm:HomoenergticSol1}, formula \eqref{eq:HilbertExpFirstOrder2}. The smallness assumptions on $ h_0 $, $ \beta_0 $ as well as the considered spaces coincide with the conditions in Theorem \ref{thm:HomoenergticSol1}. Hence, Theorem \ref{thm:MainTheorem} applies.
	
	As a consequence of \eqref{eq:SecondOrderHilbertExpDecay} and \eqref{eq:InvTempAsyBehav} we get $ \norm[L^1_2]{h_t}=\mathcal{O}((1+t)^{-2}) $. In addition, \eqref{eq:InvTempAsyBehav} implies $ \eta_t=\beta_t^{-\gamma/2} \approx 1+t $. Let us now compute the asymptotics of the inverse temperature \eqref{eq:AsymptoticsInversTempCombOrthogShear}. We plug the decomposition $ f_t=\mu+\bar{\mu}_t+h_t $ into \eqref{eq:IndtrodModelHomoengBEProof}, yielding
	\begin{align*}
		\dfrac{d}{dt}\left( \beta_t^{-\gamma/2} \right)  = \dfrac{\gamma \bar{a}}{3} - \gamma\,  \alpha_t^2 \, \beta_t^{-\gamma/2}, \quad \alpha_t^2:= \dfrac{1}{3}\int v\cdot A v \, h_t(v) \, dv,
	\end{align*}
	where $ \bar{a} $ is given in Theorem \ref{thm:HomoenergticSol1} in \eqref{eq:AsymptoticsSolCorollary1Constant}. We know that $ \beta_t^{-\gamma/2} = \mathcal{O}(t) $ and $ \alpha_t^{(2)} = \mathcal{O}((1+t)^{-2}) $ as $ t\to \infty $. Hence, we have
	\begin{align*}
			\dfrac{d}{dt}\left( \beta_t^{-\gamma/2} \right) = \dfrac{\gamma \bar{a}}{3} +R_t, \quad |R_t| \leq \dfrac{C}{1+t}.
	\end{align*}
	We integrate this in time to get \eqref{eq:AsymptoticsInversTempSimpleShear}. Finally, \eqref{eq:AsymptoticsInversTempSimpleShear} implies $ \eta_t=\beta_t\approx \beta_0^{-\gamma/2}+t=:\zeta_t $. As a consequence, we obtain \eqref{eq:FirstSecondOrderHilbertExpDecay1} from \eqref{eq:SecondOrderHilbertExpDecay} and \eqref{eq:FirstOrderHilbertExpDecay}.
	
	\textit{Simple shear with decaying planar dilatation/shear:} In this case, $ L_t $ is given by \eqref{eq:SimpShearDecPlanarDil}. We define also $ \nu_t:=\rho_t $, where the density is given in \eqref{eq:IndtrodModelHomoengBEProof}. We again check the structural conditions in Theorem \ref{thm:MainTheorem}. Assumption $ (I) $ is satisfied, since $ \sup_{t\geq 0}\norm{L_t}<\infty $, $ \nu_t'/\nu_t=\trace L_t $ and $ L_t'=-L_t^2 $. The latter equation is part of the ansatz of homoenergetic solutions in \eqref{eq:homenergBE}. For assumption $ (II) $ we apply Lemma \ref{lem:AssumptionII}. To this end, we use the decomposition of $ L_t=A_t+b_tI $ into its trace-free and trace part. Here, we have 
	\begin{align*}
		A_t&= A^0+A^1_t=\left( \begin{array}{ccc}
		0 & K_2 & 0 \\ 
		0 & 0 & 0 \\ 
		0 & 0 & 0
		\end{array} \right)+\dfrac{1}{1+t} \left( \begin{array}{ccc}
		0 & K_1K_3 & K_1 \\ 
		0 & 0 & 0 \\ 
		0 & K_3 & 0
		\end{array} \right) + \mathcal{O}\left( \dfrac{1}{(1+t)^2} \right),
		\\
		b_t&=\trace L_t = \dfrac{1}{1+t}+b^2_t, \quad b^2_t = \mathcal{O}\left( \dfrac{1}{(1+t)^2} \right).
	\end{align*}
	Note that $ A^0\neq 0 $ due to $ K_2\neq 0 $. Furthermore, we have
	\begin{align*}
		\nu_t&=\rho_t = \exp\left( -\int_0^t\trace L_s\, ds \right) \approx (1+t)^{-1},
		\\
		N_t &\approx \int_0^t\left( \dfrac{1+s}{1+t}\right)^{\gamma/3+1} \, ds =\dfrac{3}{\gamma+6}\dfrac{(1+t)^{\gamma/3+2}-1}{(1+t)^{\gamma/3+1}}\approx t.
	\end{align*}
	Hence, Lemma \ref{lem:AssumptionII} implies assumption $ (II) $ in Theorem \ref{thm:MainTheorem}. Note again that the definition of $ \bar{\mu} $ in \eqref{eq:HilbertExpFirstOrder3} is the same as in Theorem \ref{thm:HomoenergticSol1}, formula \eqref{eq:HilbertExpFirstOrder2}. 
	
	We can now apply Theorem \ref{thm:MainTheorem}. Again, \eqref{eq:SecondOrderHilbertExpDecay} and \eqref{eq:InvTempAsyBehav} yields $ \norm[L^1_2]{h_t}=\mathcal{O}((1+t)^{-2}) $. We now calculate the asymptotics for $ \beta_t $. We again have with \eqref{eq:ProofEqInverseTemp} and the decomposition $ f_t= \mu+\bar{\mu}_t+h_t $
	\begin{align*}
		\dfrac{d}{dt}\left( \beta_t^{-\gamma/2} \right)  = -\dfrac{\gamma\,  \trace L_t}{3}\beta_t^{-\gamma/2} + \dfrac{\gamma\, a_t}{3\rho_t} - \gamma \alpha_t^2\,  \beta_t^{-\gamma/2}, \quad \alpha_t^2:= \int v\cdot L_t v \, h_t(v) \, dv.
	\end{align*}
	Here, we used
	\begin{align*}
		a_t := \dualbra{v\cdot A_tv \, \mu}{\Lin^{-1}[v\cdot A_tv \, \mu]}_{L^2(\mu^{-1/2})}.
	\end{align*}
	One can see that
	\begin{align*}
		\rho_t^{-1} &= (1+t)\exp\left( \int_0^t r_s \, ds \right) = (1+t)\exp\left( \int_0^\infty r_s \, ds \right) + \mathcal{O}(1),
		\\
		a_t &= \bar{a} + \mathcal{O}\left( \dfrac{1}{1+t} \right).
	\end{align*}
	Here, $ \bar{a} $ and $ r_t $ are given in Theorem \ref{thm:HomoenergticSol1}, see \eqref{eq:AsymptoticsSolCorollary2Constant} and \eqref{eq:AsymptoticsSolCorollary22Constant}. From \eqref{eq:AsymptoticsInversTempSimpleShearDecaying} and $ \nu_t^{-1}=\rho_t^{-1} $ we get $ \beta^{-\gamma/2}=\mathcal{O}(t^2) $. Using this and $ |\alpha_t^2|\lesssim\norm[L^1_2]{h_t} = \mathcal{O}((1+t)^{-2}) $ we obtain
	\begin{align*}
		\dfrac{d}{dt}\left( \beta_t^{-\gamma/2} \right)  = -\dfrac{\gamma }{3(1+t)}\beta_t^{-\gamma/2} + \dfrac{\gamma \bar{a}}{3}(1+t)\exp\left( \int_0^\infty r_s \, ds \right)  + R_t, \quad |R_t|\leq C.
	\end{align*}
	We integrate this ODE yielding
	\begin{align}\label{eq:ProofAsymptotics}
		\beta_t^{-\gamma/2} = \beta_0^{-\gamma/2}(1+t)^{-\gamma/3} + \dfrac{\gamma \bar{a}}{\gamma+6}\exp\left( \int_0^\infty r_s \, ds \right) \left( (1+t)^{2}-1\right) + \tilde{R}_t.
	\end{align}
	Here, we have the lower order term
	\begin{align*}
		|\tilde{R}_t| \leq C \int_0^t \left( \dfrac{1+t}{1+s} \right)^{-\gamma/3}\, ds \leq C (1+t).
	\end{align*}
	Hence, we get \eqref{eq:AsymptoticsInversTempSimpleShearDecaying}. The formula \eqref{eq:ProofAsymptotics}, $ \nu_t = \rho_t\approx (1+t)^{-1} $ and \eqref{eq:AsymptoticsInversTempSimpleShearDecaying} yields
	\begin{align*}
		\eta_t=\nu_t\beta_t^{-\gamma/2} \approx \beta_0^{-\gamma/2}(1+t)^{-1-\gamma/3} + t=:\zeta_t.
	\end{align*}
	Thus, by \eqref{eq:SecondOrderHilbertExpDecay} and \eqref{eq:FirstOrderHilbertExpDecay} we obtain \eqref{eq:FirstSecondOrderHilbertExpDecay2}.
	
	\textit{Combined orthogonal shear:} Here, $ L_t $ is given by \eqref{eq:CombOrthShear}. In this case, the matrix $ L_t $ is not the matrix in equation \eqref{eq:ModelEq}. The reason is that we need to take care of the linear growth of $ L_t $, so that assumption $ (I) $ in Theorem \ref{thm:MainTheorem} is not satisfied.
	
	We have first of all $ \trace L_t=0 $ so that $ \rho_t\equiv \rho_0=1 $. In order to apply Theorem \ref{thm:MainTheorem}, let us introduce the time-change $ \tau=(t+1)^2/2-1/2 $, i.e. $ 1+t= \sqrt{2\tau+1} $ in equation \eqref{eq:IndtrodModelHomoengBEProof}. Hence, $ F(\tau,v):=f(t(\tau),v) $ solves the equation
	\begin{align*}
		\partial_\tau F &= \div\left( \left( \tilde{L}_\tau-\alpha_\tau\right) v\, F \right) + \nu_\tau  \beta_\tau^{-\gamma/2}Q(F,F)
	\end{align*}
	where we defined
	\begin{align*}
		\tilde{L}_\tau &= \dfrac{1}{\sqrt{2\tau+1}}L(t(\tau))=A^0+A^1_\tau= \left( \begin{array}{ccc}
		0 & 0 & -K_1K_3 \\ 
		0 & 0 & 0 \\ 
		0 & 0 & 0
		\end{array} \right)+ \dfrac{1}{\sqrt{2\tau+1}}\left( \begin{array}{ccc}
		0 & K_3 & K_2+K_1K_2 \\ 
		0 & 0 & K_1 \\ 
		0 & 0 & 0
		\end{array} \right),
		\\
		\nu_\tau &= \dfrac{1}{\sqrt{2\tau+1}}.
	\end{align*}
	The above equation is of the form \eqref{eq:ModelEq} and we aim to apply Theorem \ref{thm:MainTheorem}. Note that the first order approximation $ \bar{\mu}_t $ in \eqref{eq:HilbertExpFirstOrder2} can be written as 
	\begin{align*}
		\bar{\mu}_{t(\tau)} = \dfrac{1}{\beta_{t(\tau)}^{-\gamma/2}}\Lin^{-1}\left[ -v\cdot L_{t(\tau)} v \mu \right] = \dfrac{1}{\nu_{\tau}\beta_{t(\tau)}^{-\gamma/2}}\Lin^{-1}\left[ -v\cdot A_\tau v \mu \right].
	\end{align*} 
	Hence, we have \eqref{eq:HilbertExpFirstOrder3} in terms of the time $ \tau $, which is used in Theorem \ref{thm:MainTheorem}.
	
	Let us check now the assumptions in Theorem \ref{thm:MainTheorem}. First of all, $ \tilde{L} $ and $ \nu $ satisfy assumption $ (I) $ in Theorem \ref{thm:MainTheorem}. Moreover, $ \trace \tilde{L}_\tau=3b_\tau=0 $ and the formula for $ \tilde{L}_\tau=A_\tau $ yields the decomposition $ A^0+A^1_\tau $ as in Lemma \ref{lem:AssumptionII} $ (i) $. Furthermore, $ \nu $ satisfies $ (iii) $ in Lemma \ref{lem:AssumptionII}. With 
	\begin{align*}
		N_\tau = \int_0^\tau \dfrac{\nu_{\tau}}{\nu_{\sigma}}\, d\sigma =  \dfrac{2\tau+1}{3}-\dfrac{1}{3\sqrt{2\tau+1}} \approx \tau
	\end{align*}
	we can apply Lemma \ref{lem:AssumptionII} and assumption $ (II) $ in Theorem \ref{thm:MainTheorem} holds.
	
	We have the decomposition $ F_\tau= f_{t(\tau)}= \mu + \bar{\mu}_{t(\tau)}+h_{t(\tau)} $. From \eqref{eq:SecondOrderHilbertExpDecay} and \eqref{eq:InvTempAsyBehav} we have $ \beta_{t(\tau)}^{-\gamma/2} = \mathcal{O}(\tau^{3/2}) $ and $ \norm[L^1_2]{h_{t(\tau)}} =\mathcal{O}((1+\tau)^{-2}) $. Hence, with respect to the original time $ 1+t= \sqrt{1+2\tau} $ we get $ \norm[L^1_2]{h_t}=\mathcal{O}((1+t)^{-4}) $ and $ \beta_t^{-\gamma/2} = \mathcal{O}(t^3) $. Let us now compute the asymptotics of the inverse temperature $ \beta_t $. We use \eqref{eq:IndtrodModelHomoengBEProof} and the decomposition $ f_t=\mu+\bar{\mu}_t+h_t $. This leads to
	\begin{align*}
		\dfrac{d}{dt}\left( \beta_t^{-\gamma/2} \right) = \dfrac{\gamma \, a_t}{3} - \gamma \, \alpha_t^2\, \beta_t^{-\gamma/2}.
	\end{align*}
	Here, we abbreviated
	\begin{align*}
		\alpha^2_t &= \dfrac{1}{3} \int v\cdot L_t v \, 	h_t\,  dv = \mathcal{O}((1+t)^{-3}),
		\\
		a_t &= \dualbra{v\cdot L_t v \mu }{\Lin^{-1}\left[v\cdot L_t v \mu \right]}_{L^2(\mu^{-1/2})} = \bar{a}\, t^2+ \mathcal{O}\left( t \right). 
	\end{align*}
	Note that we used the form of the matrix $ L_t $ in \eqref{eq:CombOrthShear} for the last equality. The constant $ \bar{a} $ is given in \eqref{eq:AsymptoticsSolCorollary3Constant}. Together with $ \beta_t^{-\gamma/2} = \mathcal{O}(t^3) $ we obtain the equation
	\begin{align*}
		\dfrac{d}{dt}\left( \beta_t^{-\gamma/2} \right) = \dfrac{\gamma \, \bar{a} \, t^2}{3} + R_t, \quad |R_t|\leq C (1+t).
	\end{align*}
	Integrating this equation yields \eqref{eq:AsymptoticsInversTempCombOrthogShear}. Also we obtain $ \beta_t^{-\gamma/2}\approx \beta_0^{-\gamma/2}+t^3 $. Thus, we get
	\begin{align*}
		\eta_{t(\tau)}=\nu_\tau \,\beta_{t(\tau)}^{-\gamma/2}\approx \left( \beta_0^{-\gamma/2}+t(\tau)^3\right) (1+t(\tau))^{-1}
	\end{align*}
	and with respect to the original time $ t $
	\begin{align*}
		\eta_{t}\approx \beta_0^{-\gamma/2}(1+t)^{-1}+t^2=:\zeta_t.
	\end{align*}
	Finally, as a consequence of \eqref{eq:SecondOrderHilbertExpDecay} ad \eqref{eq:InvTempAsyBehav} we have
	\begin{align*}
		\norm[\mathcal{H}^1_{p_0}]{h_{t(\tau)}}\leq C'\left(  \dfrac{\varepsilon}{(1+\tau)^2}+\dfrac{16}{\eta_{t(\tau)}^2}\right).
	\end{align*}
	Using the previous estimate for $ \eta_t $ and writing this with respect to the original time $ 1+t= \sqrt{1+2\tau} $ yields \eqref{eq:FirstSecondOrderHilbertExpDecay3}. The same can be done using \eqref{eq:FirstOrderHilbertExpDecay} to get the second estimate in \eqref{eq:FirstSecondOrderHilbertExpDecay3}.
\end{proof}

\section{Collision dominated behavior for cutoff kernels}
\label{sec:Cutoff}
In this final section, we indicate an extension of the previous analysis to cutoff kernels. In particular, we consider the following assumptions.
\paragraph{Assumptions on the kernel.} The collision kernel has the product form $ B(v-v_*,\sigma)=b(n\cdot \sigma)|v-v_*|^\gamma $, where $ b:[-1,1]\rightarrow[0,\infty) $ is a smooth function and $ \gamma $ satisfies $ \gamma\in(0,1] $.

The most prominent application is the case of hard spheres interactions $ B(v-v_*,\sigma)=|v-v_*| $. Let us now give the following $ L^1 $-variant of Theorem \ref{thm:HomoenergticSol1} and Theorem \ref{thm:HomoenergticSol2}.

\begin{thm}\label{thm:HomoenergticSol3}
	Consider equation \eqref{eq:homenergBE} with matrix $ L_t =L_0(I+tL_0)^{-1}  $ having the asymptotic form \eqref{eq:SimpleShear}, \eqref{eq:SimpShearDecPlanarDil} or \eqref{eq:CombOrthShear}. Let $ p_0>3 $ be arbitrary and $ g_0\in \mathcal{W}^{1,1}_{p_0} $. Consider the unique solution $ g $ to \eqref{eq:homenergBE}. Define $ f,\, \bar{\mu} $ as in \eqref{eq:HilbertExpFirstOrder2} and $ h_t(v) := f_t(v)-\mu(v)-\bar{\mu}_t(v) $. 
	
	There are $ \varepsilon_0\in(0,1) $ sufficiently small and a constant $ C'>0 $, depending only on $ p_0 $, $ L_t $ and the collision kernel $ B $, such that: If $ \norm[W^{1,1}_{p_0}]{h_0}=\varepsilon\leq \varepsilon_0 $ and $ \beta_0\leq \varepsilon_0 $, we have in each case the asymptotics \eqref{eq:AsymptoticsInversTempCombOrthogShear}, \eqref{eq:AsymptoticsInversTempSimpleShear} and \eqref{eq:AsymptoticsInversTempSimpleShearDecaying}. Finally, the bounds in \eqref{eq:FirstSecondOrderHilbertExpDecay1}, \eqref{eq:FirstSecondOrderHilbertExpDecay2} and \eqref{eq:FirstSecondOrderHilbertExpDecay3} are true when replacing $ \norm[\mathcal{H}^1_{p_0}]{h_t} $ by $ \norm[W^{1,1}_{p_0}]{h_t} $.
\end{thm}
Let us mention that existence of (weak) solutions to \eqref{eq:homenergBE} are known by the work of Cercignani \cite{Cercignani1989ExistenceHomoenergetic}. Uniqueness can be proved as in Proposition \ref{pro:Wellposedness}, which amounts to an application of a Povzner estimate. Furthermore, Povzner estimates allow to show the gain of moments as in Proposition \ref{pro:Wellposedness} (i). The propagation of regularity estimates can be proved by arguments used for the homogeneous Boltzmann equation, see e.g. \cite{MouhotVillani2004RegularityHomogBoltzEq}.

In order to prove this theorem, we show that the corresponding $ L^1 $-variant of Theorem \ref{thm:MainTheorem} is valid. To this end, we follow the strategy in Subsection \ref{subsec:CollDomAnalysis}.

\subsection{Collision-dominated analysis for cutoff kernels}
The proof is analogous to the one discussed in Subsection \ref{subsec:CollDomAnalysis} and the goal is to prove \eqref{eq:LoopInequality1} and the variant of \eqref{eq:LoopInequality2}, namely,
\begin{align*}
	\norm[W^{1,1}_{p_0}]{h_t}\leq \Omega\left( \dfrac{\varepsilon}{(1+t)^2}+\dfrac{1}{\eta_t^2}\right).
\end{align*}
The four main ingredients in Subsection \ref{subsec:CollDomAnalysis} are Lemma \ref{lem:LoopInequality1}, Lemma \ref{lem:FirstOrderEstimate} and Proposition \ref{pro:PrepLoopInequality2}, Proposition \ref{pro:LoopInequality2}. The first lemma extends without any changes. Let us discuss the other three preparatory results. 

\subsubsection{Estimate on $ \bar{\mu} $ for cutoff kernels}
The result in Lemma \ref{lem:FirstOrderEstimate} relies on corresponding coercivity estimates for the operator $ Lg=\mu^{-1/2}\Lin[\sqrt{\mu}g] $ in the cutoff case. Such estimates in Sobolev spaces $ H^k $ were discussed in \cite{MouhotNeumann2006QuantitativePertStudyConvEquilib}. (They considered general operators $ L $ which satisfy the hypothesis H1' and H2' therein. These assumptions are exactly the needed coercivity estimates.) To prove bounds including higher moments, it suffices by interpolation to prove corresponding coercivity estimates in $ L^2_p $, i.e. for all $ p\geq0 $
\begin{align}\label{eq:CutoffCoercivity}
	\dualbra{Lg}{g}_{L^2_p}\geq c_0\norm[L^2_{p+\gamma/2}]{g}^2-C\norm[L^2]{g}^2.
\end{align}
For this one uses the commutator estimate
\begin{align*}
	\left| \dualbra{\left\langle v \right\rangle ^p\, Lg}{\left\langle v \right\rangle^p\, g}_{L^2}-\dualbra{L\left[ \left\langle v \right\rangle ^p\,g \right] }{\left\langle v \right\rangle^p\,g}_{L^2} \right| \leq \varepsilon \norm[L^2_{p+\gamma/2}]{g}^2+C_\varepsilon \norm[L^2_{m-1+\gamma/2}]{g}^2
\end{align*}
holding for all $ \varepsilon>0 $. Furthermore, we have by the spectral gap, see e.g. \cite{MouhotNeumann2006QuantitativePertStudyConvEquilib},
\begin{align*}
	\dualbra{L\left[ \left\langle v \right\rangle ^p\,g \right] }{\left\langle v \right\rangle^p\,g}_{L^2} \geq c_0\norm[L^2_{\gamma/2}]{\left( I-\Pi_0 \right)\left\langle v \right\rangle^p\,g }^2 \geq \dfrac{c_0}{2}\norm[L^2_{p+\gamma/2}]{g}^2-C\norm[L^2]{g}^2.
\end{align*}
Recall that $ \Pi_0 $ is the projection onto $ \ker L $, which is spanned by the functions $ \varphi\, \sqrt{\mu} $ with $ \varphi(v)=1,\, v_1,\, v_2,\, v_3,\, |v|^2 $. Combining the previous estimates, using interpolation and choosing $ \varepsilon $ small yields \eqref{eq:CutoffCoercivity}.

\subsubsection{Estimate on the error term for cutoff kernels}
Instead of the estimates in $ L^2 $ leading to Proposition \ref{pro:PrepLoopInequality2} we use the following bounds, which replaces Lemma \ref{lem:RegEstLinearizedCollOp} and Lemma \ref{lem:L2EstimateCollOp}. However, due to the fact that there is no regularizing effect, we need a second estimate for the linearized collision operator, which takes into account first order derivatives.

\begin{lem}\label{lem:CutoffL1Estimate}
	We have the following estimates.
	\begin{enumerate}[(i)]
		\item For $ p\geq\gamma $ it holds
			\begin{align*}
				\norm[L^1_p]{Q(f,g)}\lesssim \norm[L^1_p]{f}\norm[L^1_{p+\gamma}]{g}+\norm[L^1_{p+\gamma}]{f}\norm[L^1_p]{g}.
			\end{align*}
		\item For $ p>2 $ we have
			\begin{align*}
				-\int_{\R^3}\Lin h\, \sgn(h)\, \left\langle v \right\rangle^p\, dv \leq -c_0 \norm[L^1_{p+\gamma}]{h}+C\norm[L^1]{h}.
			\end{align*}
		\item Let $ i=1,\, 2,\, 3 $ and $ p>2 $, then we have
		\begin{align*}
					-\int_{\R^3}\partial_{i}\left[ \Lin h \right]  \, \sgn(\partial_{i} h)\, \left\langle v \right\rangle^p\, dv \leq -c_0 \norm[L^1_{p+\gamma}]{\partial_{i}h}+C\norm[L^1_{p+\gamma}]{h}.
		\end{align*}
	\end{enumerate}
\end{lem}
\begin{proof}
	The first bound can be proved via rough estimates and the second one uses a variant of the Povzner estimate. For the last inequality, we use the decomposition of $ \Lin = \mathscr{B}_\varepsilon+\mathscr{A}_\varepsilon $ in \cite{Tristani2014ExponentialConvergenceEquilHomogBoltz}, which is defined as follows. Let $ 0\leq \Theta_\varepsilon \leq 1 $ be smooth, equal to one on the set
	\begin{align*}
			\set{|v|\leq 1/\varepsilon,\; 2\varepsilon\leq |v-v_*|\leq 1/\varepsilon,\; |\cos\theta|\leq 1-2\varepsilon}
	\end{align*}
	and supported in the set
	\begin{align*}
		\set{|v|\leq 1/2\varepsilon,\; \varepsilon\leq |v-v_*|\leq 2/\varepsilon,\; |\cos\theta|\leq 1-\varepsilon}.
	\end{align*}
	One can choose it in the form $ \Theta_\varepsilon (v,v_*,\theta)=\Theta_\varepsilon^1(v)\Theta_\varepsilon^2(v-v_*)\Theta_\varepsilon^3(\theta) $. Then, we define
	\begin{align*}
			\mathscr{B}_\varepsilon h &= \int_{\R^3}\int_{S^2}(1-\Theta_\varepsilon)\, |v-v_*|^\gamma \, b(\cos\theta)\left( \mu'h'_*+\mu'_*h' - \mu h_*\right) \, d\sigma dv_*-\norm[L^1(S^2)]{b}\left( |\cdot|^\gamma * \mu \right)\, h,
			\\
			\mathscr{A}_\varepsilon h &= \int_{\R^3}\int_{S^2}\Theta_\varepsilon\, |v-v_*|^\gamma \, b(\cos\theta)\left( \mu'h'_*+\mu'_*h' - \mu h_*\right) \, d\sigma dv_*.
	\end{align*}
	To calculate the derivative with respect to $ v_i $ we use the change of variables $ w=v-v_* $ in the $ v_* $-integration, perform the differentiation and undo the transformation again. We obtain in this way 
	\begin{align*}
			\partial_{i}\left[ \Lin h \right] &= \partial_{i}\left[ \mathscr{B}_\varepsilon h\right] + \partial_{i}\left[ \mathscr{A}_\varepsilon h\right]
			=  \mathscr{B}_\varepsilon \partial_{i}h+ \partial_{i}\left[ \mathscr{A}_\varepsilon h\right]-\norm[L^1(S^2)]{b}\left( |\cdot|^\gamma * \partial_{i}\mu \right)\, h
			\\
			&+\int_{\R^3}\int_{S^2}(1-\Theta_\varepsilon)\, |v-v_*|^\gamma \, b(\cos\theta)\left( (\partial_{i}\mu)'h'_*+(\partial_{i}\mu)'_*h' - (\partial_{i}\mu) h_*\right) \, d\sigma dv_*
			\\
			&-\int_{\R^3}\int_{S^2}\left[ \partial_{i}\Theta_\varepsilon^1\right] \Theta_\varepsilon^2\Theta_\varepsilon^3\, |v-v_*|^\gamma \, b(\cos\theta)\left( (\partial_{i}\mu)'h'_*+(\partial_{i}\mu)'_*h' - \partial_{i}\mu h_*\right) \, d\sigma dv_*.
	\end{align*}
	The last three terms can be estimated in $ L^1_p $ by $ C_\varepsilon\norm[L^1_{p+\gamma}]{h} $. The first term is strongly dissipative in $ L^1_p $, $ p>2 $, for $ \varepsilon>0 $ sufficiently small. For this see the proof of \cite[Lemma 2.6]{Tristani2014ExponentialConvergenceEquilHomogBoltz}, which covers the non-cutoff case and uses a variant of the Povzner estimate. In the cutoff case, there is no need to split $ b $ into a cutoff and non-cutoff part. We then obtain
	\begin{align*}
			\int_{\R^3}\, \mathscr{B}_\varepsilon \partial_{i} h \, \sgn(\partial_{i} h)\, \left\langle v \right\rangle^p\, dv \leq -c_0 \norm[L^1_{p+\gamma}]{\partial_{i} h}.
	\end{align*}
	Finally, the operator $ \mathscr{A}_\varepsilon $ is regularizing, in the sense that it maps $ L^1_1 $ to compactly supported functions and
	\begin{align*}
			\norm[H^1]{\mathscr{A}_\varepsilon h}\leq C_\varepsilon \norm[L^1_1]{h}.
	\end{align*}
	For this result see \cite[Lemma 4.16]{Mouhot2017FactorizationNonSymOper}. This relies on the regularizing effect of the gain term, see \cite[Theorem 3.1]{MouhotVillani2004RegularityHomogBoltzEq} and references therein.
	Putting all estimates together yields the result.
\end{proof}
\paragraph{Sketch of estimates.} We state here the estimates for the corresponding $ L^1 $-variant of Proposition \ref{pro:PrepLoopInequality2} and Proposition \ref{pro:LoopInequality2}. Let us mention that compared to Subsection \ref{subsec:CollDomAnalysis}, cf. \eqref{eq:Defintionm}, we define here $ m:=p_0-1>2 $. Furthermore, we give here only a priori estimates. In particular, when it comes to the continuation argument the continuity of  $ t\mapsto\norm[W^{1,1}_{p_0}]{h_t} $ is crucial. A way to ensure this is to regularize the initial data $ h^n_0 $ with $ \norm[W^{1,1}_{p_0}]{h^n_0}\leq 2\varepsilon $. For the corresponding solution we can prove the $ L^1 $-variant of \eqref{eq:SecondOrderHilbertExpDecay} as well as \eqref{eq:InvTempAsyBehav} and pass to the limit.

Concerning the proof of Proposition \ref{pro:PrepLoopInequality2} one uses again an equivalent norm (we again abbreviate $ p=p_0 $)
\begin{align*}
	\tnorm[W^{1,1}_p]{h} := \norm[L^1_p]{h}+\kappa\sum_{|\alpha|= 1}\norm[L^1_p]{\partial^\alpha h}
\end{align*}
for $ \kappa\in (0,1) $. We estimate 
\begin{align*}
		\dfrac{d}{dt}\tnorm[W^{1,1}_p]{h} = \int_{\R^3} \partial_th\; \sgn(h) \, \left\langle  v \right\rangle^p\, dv + \kappa \sum_{i=1}^3\int_{\R^3} \partial_i\partial_th\; \sgn(\partial_i h) \, \left\langle  v \right\rangle^p\, dv.
\end{align*}
The first term can be treated similar as in the proof of Proposition \ref{pro:PrepLoopInequality2} using Lemma \ref{lem:CutoffL1Estimate} (i), (ii). For the derivatives, the most important change appears for the linearized collision operator  (compare Step 1 (vi) in the proof of Proposition \ref{pro:PrepLoopInequality2}), where we use Lemma \ref{lem:CutoffL1Estimate} (iii). Here, the lower order term $ \kappa\eta_tC \norm[L^1_{p+\gamma}]{h} $  can be absorbed into the term $ -c_0\eta_t\norm[L^1_{p+\gamma}]{h} $, resulting from Lemma \ref{lem:CutoffL1Estimate} (ii), for $ \kappa>0 $ small enough. All in all, one obtains
\begin{align*}
	\dfrac{d}{dt}\tnorm[W^{1,1}_p]{h} = -\left(\dfrac{c_0\eta_t}{2}-C-\eta_tC\tnorm[W^{1,1}_p]{h}  \right) \tnorm[W^{1,1}_p]{h} + \dfrac{C'}{\eta_t} + \eta_t C'\norm[L^1]{h}.
\end{align*}
The second term is a consequence of the source and the last term is the remaining term in the estimate of Lemma \ref{lem:CutoffL1Estimate} (ii). We can again choose $ \beta_0\leq \varepsilon_0 $ small enough, so that $ \eta_t\gtrsim \bar{\eta}(\beta_0) $ is large enough to absorb the constant $ C $. Following Step 2 and Step 3 in the proof of Proposition \ref{pro:PrepLoopInequality2}, we can conclude the corresponding $ L^1 $-variant of Proposition \ref{pro:PrepLoopInequality2}.

Concerning Proposition \ref{pro:LoopInequality2} we can use the arguments without any changes, noting that the collision operator can be bounded via Lemma \ref{lem:CutoffL1Estimate} (i). The drift term is estimated via the $ W^{1,1}_{p_0} $-norm, recalling $ p_0= m+1 $. A key ingredient is Lemma \ref{lem:LinearDecay}, which was proved in \cite{Tristani2014ExponentialConvergenceEquilHomogBoltz} for the cutoff case. The same proof can be used to show the result for cutoff kernels. In fact, the proof simplifies, since a decomposition $ b=b_\delta+b_\delta^c $ into a cutoff and non-cutoff part is not needed.

Finally, the conclusion based on a continuation argument does not change.
All in all, the corresponding variant of Theorem \ref{thm:MainTheorem} holds true and the same arguments as in Section \ref{sec:Application} conclude the proof of Theorem \ref{thm:HomoenergticSol3}.

\bibliographystyle{habbrv}
\bibliography{Collision_dominated_hard_potentials}

\end{document}